\documentclass[a4paper,10pt,leqno]{article}

\usepackage{amssymb,amsmath} %
\usepackage[abbrev]{amsrefs}
\usepackage{xy}
\xyoption{all}
\usepackage{graphicx}

\numberwithin{equation}{section}
\newtheorem{theorem}{Theorem}[section]
\newtheorem{lemma}[theorem]{Lemma}
\newtheorem{corollary}[theorem]{Corollary}

\newtheorem{definition}[theorem]{Definition}

\newtheorem{exm}[theorem]{Example}
\newtheorem{xproof}{{\it Proof. }} %
\newtheorem{xrem}{Remark.}

\newenvironment{proof}{\begin{xproof}\em}{\end{xproof}} %
\newenvironment{remark*}{\begin{xrem}\em}{\end{xrem}}

\def\qedhere{\hspace{0.3cm}{\rule{1ex}{2ex}}} %

\newcommand\Max{\operatorname{Max}}

\newcommand\clsets{\operatorname{\mathfrak C}}

\newcommand\Loc{\textit{Loc}}
\newcommand\LLoc{\textit{LinLoc}}
\newcommand\LinCoLoc{\textit{LinCoLoc}}

\newcommand\Top{\textit{Top}}

\newcommand\opp{\mathrm{op}}

\newcommand\ident{\mathrm{id}}

\newcommand\CC{\mathbb{C}}

\newcommand\supp{\operatorname{supp}}

\newcommand\Sub{\operatorname{Sub}}

\newcommand\linspan[1]{\left\langle #1\right\rangle}

\newcommand\QVBun{\textit{QVBun}}
\newcommand\PsiQVBun{\ensuremath{\varPsi}\textit{QVBun}}
\newcommand\PsiLLoc{\ensuremath{\varPsi}\textit{LinLoc}}
\newcommand\PsiLinCoLoc{\ensuremath{\varPsi}\textit{LinCoLoc}}
\newcommand\AQVBun{\textit{AQVBun}}
\newcommand\ALLoc{\textit{ALinLoc}}
\newcommand\ALinCoLoc{\textit{ALinCoLoc}}

\newcommand\image{\operatorname{Im}}

\newcommand\interior{\operatorname{int}}
\newcommand\sob{\operatorname{sob}}
\newcommand\spat{\operatorname{spat}}
\newcommand\kfrak{\mathfrak k}

\newcommand\Iota{\mathfrak I}
\newcommand\CHECK[1]{\overset{\scriptscriptstyle\vee}{#1}}
\newcommand\RE{\operatorname{Re}}

\newcommand\Tt{\mathcal T}

\newcommand\hide[1]{}

\begin{document}

\title{The codual quotient vector bundle}
\author{Jo\~ao Paulo Santos}
\maketitle

\begin{abstract}
  Given a quotient vector bundle $\mathcal A$ over $X$ with kernel map
  $\kappa\colon X\to\Max A$ we study the codual bundle with fiber at each
  point $x\in X$ isomorphic to the dual of $\kappa(x)$. Applying
  the adjunction
  between quotient vector bundles and linearized locales to the
  codual bundle leads to a new adjunction between the category of all
  quotient vector bundles and a category whose objects are locales $L$
  equipped with a contravariant morphism $L\to\Max A$. 
\end{abstract}

\section{Introduction}

Let $A$ be a topological vector space, 
which we will always assume is locally convex.
Let $X$ be a topological space. For each $x\in X$ choose a subspace $\kappa(x)\subset A$
and let $E=\coprod_{x\in X}A/\kappa(x)$, topologized as a quotient of $A\times X$. We say
the triple $\mathcal A=(X,A,\kappa)$
is a \emph{quotient vector bundle} \cite[section 3]{ReSa17} if the quotient map $A\times X\to E$ is open.
If we let $\Sub A$ denote the collection of all the subspaces of $A$ 
topologized with the lower Vietoris topology \cites{Vie22,NoSh96},
all quotient vector bundles can be obtained as pullbacks of the universal quotient vector bundle
$\mathcal A_U=(\Sub A,A,\ident)$ over $\Sub A$.
In \cite{ReSa16} the notion of 
spectral quotient vector bundle was introduced, together with
an adjunction $\boldsymbol\Omega\dashv\boldsymbol\Sigma$
between the category of spectral quotient vector bundles and a certain category
of linearized locales, whose objects are locales $L$ equipped with an inf-lattice homomorphism
$L\to\Sub A$ whose restriction to the spectrum $\Sigma L$ is continuous.
This adjunction extends the usual adjunction between topological spaces and locales.
It turns out, however, that the universal bundle $\mathcal A_U$ is not spectral so in general there is no universal spectral vector bundle.

In this paper we introduce the category
of \emph{linearized colocales} whose objects can be thought of as locales equipped with a
sup-lattice homomorphism $L^\opp\to\Sub A$ whose restriction to the spectrum is continuous 
(see section~\ref{sec:deflincoloc}). Then we define
an adunction $\boldsymbol\clsets\dashv\boldsymbol\Iota$
between the category of all quotient vector bundles with Hausdorff fibers and the category
of linearized colocales (Theorem~\ref{T5.10}). 
The adjunction $\boldsymbol\clsets\dashv\boldsymbol\Iota$ also extends the adjunction between topological spaces and locales and
the unit of the adjunction, $\mathcal A\to\boldsymbol\Iota\boldsymbol\clsets\mathcal A$ coincides, when $\mathcal A$ is spectral, with
the soberification map $\sob_{\mathcal A}\colon\mathcal A\to\boldsymbol\Sigma\boldsymbol\Omega\mathcal A$
defined in \cite[section~5.16]{ReSa16} under a natural isomorphism $\boldsymbol\Iota\boldsymbol\clsets\mathcal A\cong\boldsymbol\Sigma\boldsymbol\Omega\mathcal A$ (Theorem~\ref{T4.12}).

We can interpret the adjunction $\boldsymbol\clsets\dashv\boldsymbol\Iota$
in terms of a bundle over $X$ with fiber at each point $x\in X$ isomorphic to the dual space of $\kappa(x)$:
Given a quotient vector bundle $\mathcal A=(X,A,\kappa)$,
let $A^\vee$ be the dual space of continuous linear functionals and for each $x\in X$ let $\kappa^\vee(x)\subset A^\vee$ 
denote the annihilator of $\kappa(x)$.
Then $A^\vee/\kappa^\vee(x)\cong\bigl(\kappa(x)\bigr)^\vee$.
We call $E^\vee=\coprod_{x\in X}A^\vee/\kappa^\vee(x)$ the codual bundle, topologized as a quotient of $A^\vee\times X$
and we write $\mathcal A^\vee=(X,A^\vee,\kappa^\vee)$. To each linearized locale $\mathfrak A$ we associate a codual linearized colocale
$\mathfrak A^\vee$, to each linearized colocale we associate a codual linearized locale and
we show that $\boldsymbol\Omega(\mathcal A^\vee)\cong(\boldsymbol\clsets\mathcal A)^\vee$ and, if $\mathcal A$ is spectral, 
$\bigl(\boldsymbol\Omega\mathcal A\bigr)^\vee\cong\boldsymbol\clsets\bigl(\mathcal A^\vee\bigr)$ (Lemma~\ref{L6.9}).

In general the codual bundle $\mathcal A^\vee$ is not a quotient vector bundle so
we introduce the notion of pseudo quotient vector bundle by dropping the condition that the map $A\times X\to E$ be open. We
also introduce
the notions of pseudo linearized locale and colocale by dropping the continuity condition on the homomorphisms $L\to\Sub A$.
The adjunctions $\boldsymbol\Omega\dashv\boldsymbol\Sigma$ and $\boldsymbol\clsets\dashv\boldsymbol\Iota$ can be easily extended
to this setting.

The collection of closed subspaces $\Max A\subset\Sub A$.
is to a large extent independent of the topology on $A$. 
In particular, it is completely determined by the dual space $A^\vee$ of continuous linear functionals.
This leads us to the notion of algebraic quotient vector bundles, defined in terms of a pair of vector spaces
$(A,A^\vee)$ together with a non-degenerate bilinear form $\langle,\rangle\colon A\times A^\vee\to\mathbb C$. 
The category of algebraic quotient vector bundles is equivalent to the category of pseudo quotient vector bundles
and algebraic quotient vector bundles are more natural when we study the duality $\mathcal A\to\mathcal A^\vee$.
It turns out the assignements $\mathcal A\mapsto\mathcal A^\vee$ and $\mathfrak A\mapsto\mathfrak A^\vee$ don't preserve morphisms,
so we introduce the notion of covariant morphisms (as opposed to the morphisms introduced in \cite{ReSa16}, which are contravariant)
and show that the adjunction $\boldsymbol\Omega\dashv\boldsymbol\Sigma$ also holds for covariant morphisms (Theorem~\ref{T5.15}).
We define duality isomorphisms between the categories of quotient vector bundles with contavariant morphisms
and with covariant morphisms (Lemma~\ref{L6.10}) and show that, under these duality functors, we have natural isomorphisms between
the functors $\boldsymbol\Omega$ and $\boldsymbol\clsets$ (Theorem~\ref{T6.12}).
The same construction is carried out for linearized locales and linearized colocales.

In the remainder of the paper we analyze when is $\mathcal A^\vee$ a quotient vector bundle,
and we study the relationship between the properties of $\mathcal A$
and those of $\mathcal A^\vee$.
In particular 
we show that, under mild conditions, if $\mathcal A$ is a Hausdorff quotient vector bundle then $\mathcal A^\vee$ is also a Hausdorff quotient vector bundle
(Corollary~\ref{coro:Fell<=>Fell}) and if $A$ is a Banach space and $\mathcal A^\vee$ is a quotient vector bundle, then both
$\mathcal A$ and $\mathcal A^\vee$ are Banach bundles (Corollary~\ref{C7.16}).

\subsubsection{Organization of the paper}

In section~\ref{sec:2} we recall some basic definitions and facts.
In section~\ref{sec:3} we define pseudo quotient vector bundles, pseudo linearized locales and pseudo linearized colocales and extend
the adjunctions in \cite{ReSa16} to this setting.
In section~\ref{sec:4} we introduce the notion of linearized colocale, we define
the functors $\boldsymbol\clsets$ and $\boldsymbol\Iota$ 
on objects and prove some of their properties. We end the section by studying the universal bundle $\mathcal A_U$.
In section~\ref{sec:5} we introduce covariant morphisms, define the category of linearized colocales and prove the adjunctions
$\boldsymbol\Omega\dashv\boldsymbol\Sigma$ and $\boldsymbol\clsets\dashv\boldsymbol\Iota$.
In section~\ref{sec:6} we introduce the notion of algebraic quotient vector bundles and define duality isomorphisms between the covariant and
contravariant categories.
In section~\ref{sec:7} we investigate when is the codual bundle $\mathcal A^\vee$ a quotient vector bundle.

\section{Preliminaries}\label{sec:2}

To fix notation, we gather here some terminology and some simple facts which we'll be using throughout the paper.
More details can be found in \cites{Joh86,JoTi84,ReSa16}.

\subsection{Complete lattices and adjunctions}

A complete lattice $L$ is a partially ordered set in which every nonempty subset $S\subset L$ has a join (supremum)
which we denote by $\bigvee S$. It follows that every nonempty subset $S\subset L$ has also a meet (infimum), denoted by $\bigwedge S$.
We write $1=\bigvee L$ and $0=\bigwedge L$.
Given complete lattices $M$ and $L$, we say a function $f\colon M\to L$ is
a \emph{sup-lattice homomorphism} if it preserves all joins, that is, if
$f\bigl(\bigvee_\alpha b_\alpha\bigr)=\bigvee_\alpha f(b_\alpha)$
for any family $\{b_\alpha\}$ in $L$; similarly, we say $f$ is an  \emph{inf-lattice homomorphism} if it preserves all meets.

Given monotone functions $f\colon M\to L$ and $g\colon L\to M$, we say that $g$ is right adjoint to $f$ (equivalently, that $f$ is left adjoint to $g$) if
$f(x)\leq y\Leftrightarrow x\leq g(y)$ for all $x\in M$, $y\in L$. This condition is equivalent to the condition that, for all $x\in M$, $y\in L$
we have $f(g(y))\leq y$ and $x\leq g(f(x))$.
The existence of the right adjoint $g$ is equivalent to $f$ being a
sup-lattice homomorphism. Similarly,
$g$ is an inf-lattice homomorphism if and only if $g$ has a left adjoint $f$.

The following examples will play a central role in our work:
\begin{enumerate}
\item Given a vector space $A$, let $\Sub A$ denote the complete lattice of
  subspaces of $A$ under inclusion. The join of a family $\{V_\alpha\}$ in $\Sub A$ is the linear span of the union: $\bigvee_\alpha V_\alpha=\linspan{\,\bigcup_\alpha V_\alpha}$,
  and the meet is $\bigwedge_\alpha V_\alpha=\bigcap_\alpha V_\alpha$.
  A linear map $f\colon A\to B$ induces a sup-lattice homomorphism $\Sub f\colon\Sub A\to\Sub B$
  defined by $(\Sub f)(V)=f(V)$ whose right adjoint is the inverse image map $f^{-1}$.
\item Given a topological vector space $A$,
  let $\Max A\subset\Sub A$ denote the complete lattice of closed subspaces of $A$. The meets are the intersections and
  $\bigvee_\alpha V_\alpha=\overline{\linspan{\bigcup_\alpha V_\alpha}}$. A continuous linear map $f\colon A\to B$ induces a
  sup-lattice homomorphism $\Max f\colon\Max A\to\Max B$ defined by $(\Max f)(V)=\overline{f(V)}$ whose right adjoint is the inverse image map $f^{-1}$.
\item Given a topological space $X$, the collection of open sets $\Omega X$ is a complete lattice under inclusion. The join of a family is the union
  and the meet is the interior of the intersection. A map $f\colon X\to Y$ induces a
  sup-lattice homomorphism
  $f^{-1}\colon \Omega Y\to\Omega X$ whose right adjoint
  $f_!\colon\Omega X\to\Omega Y$
  is the map $f_!(U)=Y\setminus\overline{f(X\setminus U)}$.
\item Given a topological space $X$, the collection of closed sets $\clsets X$ is a complete lattice under inclusion. The meets are the
  intersections and the joins are the closure of the unions.
  A map $f\colon X\to Y$ induces an inf-lattice homomorphism $f^{-1}\colon \clsets Y\to\clsets X$ whose left adjoint
  $f_!^\complement\colon\clsets X\to\clsets Y$ is the map 
  \begin{equation}\label{eq:f!c}f_!^\complement(C)=\overline{f(C)}\end{equation}.
\item Given a lattice $L$ denote by $L^\opp$ the lattice with the same points as $L$ but with the opposite order relation.
  Any sup-lattice homomorphism $f\colon M\to L$ can be seen as 
  an inf-lattice homomorphism $f\colon M^\opp\to L^\opp$.
\end{enumerate}

\subsection{Locales}

A \emph{locale} is a complete lattice $L$ such that
$a\wedge\bigl(\bigvee_{\!\alpha} b_\alpha\bigr)=\bigvee_{\!\alpha}(a\wedge b_\alpha)$ for any $a\in L$ and any family $\{b_\alpha\}$ in $L$.
The main example of a locale is the topology of a topological space. Given locales $M$ and $L$,
a sup-lattice homomorphism $f\colon M\to L$ is said to be a \emph{homomorphism of locales} if it also preserves finite meets: $f(a\wedge b)=f(a)\wedge f(b)$ and $f(1_M)=1_L$.
We represent by \Loc\ the category of locales whose morphisms we now describe:
an arrow $f\colon L\to M$ in \Loc, called a \emph{map of locales},
is a homomorphism of locales in the opposite direction: $f^*\colon M\to L$,
which we call the \emph{inverse image homomorphism} of the map $f$.
We represent by $f_*\colon L\to M$ the right adjoint of $f^*$, and call it the \emph{direct image homomorphism} of the map $f$.

The assignement $X\mapsto\Omega X$ extends to a
functor $\Omega\colon\Top\to\Loc$: given a map $f\colon X\to Y$
the inverse image homomorphism of $\Omega f$ is
$(\Omega f)^*=f^{-1}\colon\Omega Y\to\Omega X$. There is also a functor
$\Sigma\colon\Loc\to\Top$
which we now describe. Given a locale $L$,
let $\Sigma L$ be the set of prime elements in $L$,
that is, the elements $p\neq 1$ such that, for all $a,b\in L$,
\[
a\wedge b\leq p\ \Rightarrow a\leq p\text{ or }b\leq p\,.
\]
Given $a\in L$ let
\[
U_a=\bigl\{p\in\Sigma L:a\not\leq p\bigr\}\,.
\]
Then the collection $\Omega\Sigma L=\{U_a\}_{a\in L}$ is a topology on $\Sigma L$, making $\Sigma L$ into a topological space.
Given a map of locales $f\colon L\to M$, let $f_*\colon L\to M$ be its
direct image homomorphism.
Then $f_*(\Sigma L)\subset\Sigma M$ and the map
$\Sigma f\colon\Sigma L\to\Sigma M$ equals
the restriction of $f_*$ to $\Sigma L$.
The assignement $a\mapsto U_a$  is a homomorphism of locales defining a map of locales
$\spat_L\colon\Omega\Sigma L\to L$ called the \emph{spatialization}, with $\spat_L^*(a)=U_a$.
A locale is said to be \emph{spatial} if spatialization is an isomorphism.

Given a topological space $X$, there is a continuous map
$\sob_X\colon X\to\Sigma\Omega X$ called the \emph{soberification} of $X$,
defined by $\sob_X(x)=X\setminus\overline{\{x\}}$. We say $X$ is a \emph{sober}
space if $\sob_X$ is a homeomorphism.
The functor $\Sigma$ is right adjoint to the functor $\Omega$ and this adjunction restricts to an equivalence between the full subcategories
of spatial locales and sober spaces.

\subsection{The lower Vietoris topology}

Let $A$ be a topological vector space.
For each open set $U\subset A$ let $\widetilde U=\bigl\{V\in\Sub A:V\cap U\neq\emptyset\bigr\}$.
The lower Vietoris topology \cites{Vie22,NoSh96} on $\Sub A$ is the topology generated by the collection
$\{\widetilde U\}_{U\in\Omega A}$. Open sets in the lower Vietoris topology are upward closed, that is,
if $\mathcal U\subset\Sub A$ is open and $V\in\mathcal U$,
then $W\in\mathcal U$ for any $W\in\Sub A$ such that $V\subset W$.
Given $V_0,V_1\in\Sub A$ with $V_0\subset V_1$ let
\begin{equation}\label{eq:mathcalU}
\mathcal U_{V_0,V_1}=\{W\in\Sub A:V_0\not\subset W\text{ or }W\not\subset V_1\}
\end{equation}

\begin{lemma}
The set $\mathcal U_{V_0,V_1}$ is open if and only if $V_0=0$ and $V_1\in\Max A$.
\end{lemma}
\begin{proof}
  If $V_0\neq0$ then $0\in\mathcal U_{V_0,V_1}$ and $0\subset V_0$ but
  $V_0\notin\mathcal U_{V_0,V_1}$ so $\mathcal U_{V_0,V_1}$ is not upward closed so it is not open.
  If $V_1\notin\Max A$ then $\overline V_1\in\mathcal U_{V_0,V_1}$ but
  $V_1\notin\mathcal U_{V_0,V_1}$
which shows $\mathcal U_{V_0,V_1}$ is not open since $V_1$ and $\overline V_1$ have the same neighbourhoods. Conversely,
if $V_0=0$ and $V_1\in\Max A$ then $U=A\setminus V_1$ is open in $A$ so
$\mathcal U_{0,V_1}=\widetilde U$ is open in the lower Vietoris topology.\qedhere
\end{proof}

Let $\sob\colon\Sub A\to\Sigma\Omega\Sub A$ be the soberification map.
Since every prime open set is a union of subbasic open sets and since for any family of open sets $\{U_\alpha\}$ we have
$\widetilde{\bigcup U_\alpha}=\bigcup\widetilde U_\alpha$, every prime open set is of the form $\widetilde U$ for some open set $U\subset A$.
So we must have
\begin{equation}\label{eq:soblV(V)=U_0,V}
  \sob(V)=\widetilde{A\setminus\overline V}=\mathcal U_{0,\overline V}
\end{equation}
It was shown in \cite[Theorem 6.16]{ReSa16} that,
if $A$ is locally convex, the space $\Max A\subset\Sub A$ with the subspace topology is a sober space. We will represent the soberification map
by $\sob_{\Max A}\colon\Max A\xrightarrow{\cong}\Sigma\Omega\Max A$.

\subsection{Quotient vector bundles}

A \emph{linear bundle} over a topological space $X$ is a collection of vector spaces parametrized by $X$, $\{E_x\}_{x\in X}$, together
with a topology on $E=\coprod_{x\in X}E_x$ such that the projection $\pi\colon E\to X$,
sum, scalar multiplication and the zero section  are continuous. 
A contravariant morphism of linear bundles $(E\to X)\to(F\to Y)$
is a continuous map $f_\flat\colon X\to Y$ and, for each $x\in X$, a linear map $\phi_x\colon F_{f_\flat(x)}\to E_x$
such that the map $\phi\colon F\times_YX\to E$ induced by the collection $\{\phi_x\}$ is continuous.

A \emph{quotient vector bundle} 
is a triple $\mathcal A=(X,A,\kappa)$ where $X$ is a topological space,
$A$ is a topological vector space and $\kappa\colon X\to\Sub A$ is a map which is continuous relative to the
lower Vietoris topology on $\Sub A$. We call $\kappa$ the \emph{kernel map}.
Given a quotient vector bundle $\mathcal A=(X,A,\kappa)$ let $E$ be the quotient of $A\times X$ by the equivalence
relation which identifies $(a,x)\sim(b,x)$ whenever $a-b\in\kappa(x)$. 
It was shown in \cite[Theorem~5.3]{ReSa17} that the continuity of $\kappa$ is equivalent to the
condition that the quotient map $q\colon A\times X\to E$ be open. The
induced map $\pi\colon E\to X$ defines a linear bundle over $X$
with fiber $E_x$ at each point $x\in X$ isomorphic to $A/\kappa(x)$.
We call $E$ the \emph{associated linear bundle} of the quotient vector bundle $(X,A,\kappa)$.

A (contravariant) morphism of quotient vector bundles $f\colon(X,A,\kappa_A)\to(Y,B,\kappa_B)$
is a pair $f=(f_\flat,f^\sharp)$ where $f_\flat\colon X\to Y$ is a continuous map and
$f^\sharp \colon B\to A$ is a continuous linear map, obeying, for all $b\in B$ and $x\in X$, the condition:
\begin{equation}\label{eq:morQVBun}
b\in \kappa_B\bigl(f_\flat(x)\bigr)\ \Rightarrow\ f^\sharp (b)\in \kappa_A(x)\,.
\end{equation}
This condition implies $f^\sharp$ induces a linear map between the fibers $B/\kappa_B\bigl(f_\flat(x)\bigr)$ and $A/\kappa_A(x)$,
and hence a morphism of the associated linear bundles.
We say $f$ is a strict morphism if we have an equivalence:
\begin{equation}\label{eq:strictmorQVBun}
b\in \kappa_B\bigl(f_\flat(x)\bigr)\ \Leftrightarrow\ f^\sharp (b)\in \kappa_A(x)\,.
\end{equation}
We represent the category of quotient vector bundles by $\QVBun^*$.

\subsubsection{Bundles with Hausdorff fibers}

We say a quotient vector bundle $(X,A,\kappa)$ has Hausdorff fibers if the fibers $E_x$ of the associated linear bundle $E\to X$
are Hausdorff spaces. This happens if and only if for any $x\in X$ the set $\{0\}\subset E_x$ is closed, so $(X,A,\kappa)$
has Hausdorff fibers if and only if the kernel map $\kappa$ factors
through $\Max A\subset\Sub A$. 
We represent the category of quotient vector bundles with Hausdorff fibers by $\QVBun^{1*}$.

\subsubsection{The Fell topology}

For each compact set $K\subset A$ let $\CHECK K=\{V\in\Max A:V\cap K=\emptyset\}$. The \emph{Fell topology} \cite{NoSh96}
is the coarsest topology on $\Max A$ containing both the lower Vietoris topology and all the sets $\CHECK K$.
Given a topological space $X$, we say a map $\kappa\colon X\to\Max A$ is \emph{Fell continuous}
if it is continuous when $\Max A$ is given the Fell topology.
Let $(X,A,\kappa)$ be a quotient vector bundle.
It was shown in \cite[Theorem~5.7]{ReSa17} that, if $X$ and $A$ are both first countable and Hausdorff, then
the kernel map $\kappa$ is Fell continuous if and only if the associated  linear bundle $E$ is Hausdorff.

\subsubsection{Normed quotient vector bundles}

We say a quotient vector bundle $\mathcal A=(X,A,\kappa)$ is \emph{normed} if
$A$ is a normed vector space. The norm on $A$ induces a norm on the associated linear bundle $E$
defined by $\|q(a,x)\|=\text{dist}(a,\kappa(x))$. If this norm is continuous we say the normed quotient vector bundle $\mathcal A$ is \emph{continuous}.
Let $A$ be a normed vector space. For each $a\in A$ and $r>0$ let $\mathbf U_r(a)=\{V\in\Max A:d(a,V)>r\}$.
The \emph{closed balls topology} is the coarsest topology on $\Max A$ containing both the lower Vietoris topology and the sets $\mathbf U_r(a)$.
Let $\mathcal A=(X,A,\kappa)$ be a normed quotient vector bundle.
It was shown in \cite[Theorem~5.12]{ReSa17} that the map $\kappa\colon X\to\Max A$ is continuous with respect to the closed balls topology if and only if
$\mathcal A$ is continuous.  If $A$ is Banach and $X$ is first countable, it was also shown that the associated bundle $E$
is a Banach bundle if and only if $\kappa$ is continuous with respect to the closed balls topology.

\subsection{Spectral quotient vector bundles}\label{sec:spectralQVB}

\subsubsection{The open support property}

Let $(X,A,\kappa)$ be a quotient vector bundle.
Each $a\in A$ gives rise to a section $\widehat a$
of the associated linear bundle $E$, namely:
$\widehat a(x)=a+\kappa(x)\in E_x$. Clearly $\widehat a(x)\neq 0\Leftrightarrow a\notin\kappa(x)$.
The open support of $\widehat a$ is the open set
\begin{equation}\label{eq:defopensupp}
\supp^\circ\widehat a=\text{int}\{x\in X:a\notin\kappa(x)\}\,.
\end{equation}
Given a subset $C\subset A$ represent by $\linspan C$ be the linear span of $C$
and for each $a\in A$ let
\begin{equation}\label{eq:defnofchecka}
\CHECK a=\{V\in\Sub A:a\notin V\}=\mathcal U_{\linspan a,1}
\end{equation}
(see equation~\eqref{eq:mathcalU}).
We say the quotient vector bundle $\mathcal A=(X,A,\kappa)$ has the \emph{open support property} if for all $a\in A$
the set $\kappa^{-1}(\CHECK a)=\{x\in X:a\notin\kappa(x)\}$ is an open set, in which case $\supp^\circ\widehat a=\kappa^{-1}(\CHECK a)$.
Quotient vector bundles over $X$ with the open support property are classified by kernel maps
$\kappa\colon X\to\Sub A$ continuous with respect to the open support topology,
which is the coarsest topology containing both the lower Vietoris topology and all the sets $\CHECK a$ (see \cite[Lemma~6.2]{ReSa16}).

\begin{lemma}\label{lemma:bigcupCHECKa}
For any set $C\subset A$ we have
\[
\bigcup_{a\in C}\CHECK a=\mathcal U_{\linspan C,1}
\]
\end{lemma}
\begin{proof}
  For any $C\subset A$ we have:
  \[
  V\in\mathcal U_{\linspan C,1}\ \Leftrightarrow\ \linspan C\not\subset V\ \Leftrightarrow\ C\not\subset V\ \Leftrightarrow\
  \exists_{a\in C}:a\notin V\ \Leftrightarrow\ V\in\bigcup_{a\in C}\CHECK a\,.\qedhere
  \]
\end{proof}

\begin{lemma}\label{lemmaOpenSupportiff}
  A quotient vector bundle $\mathcal A=(X,A,\kappa)$ has the open support property
  if and only if $\sigma(V)=\kappa^{-1}(\mathcal U_{V,1})$ for any $V\in\Sub A$.
\end{lemma}
\begin{proof}
  If $\sigma(V)=\kappa^{-1}(\mathcal U_{V,1})$ for any $V\in\Sub A$ then the set
  \[
  \kappa^{-1}(\CHECK a)=\bigl\{x\in X:a\notin\kappa(x)\bigr\}=\kappa^{-1}\bigl(\mathcal U_{\linspan a,1}\bigr)
  =\sigma\bigl(\linspan a\bigr) %
  \]
  is an open set so $\mathcal A$ has the open support property. Conversely, if $\mathcal A$ has the open support property then
  $\supp^\circ\widehat a=\kappa^{-1}(\CHECK a)$ so from Lemma~\ref{lemma:bigcupCHECKa} we get:
  \[
  \sigma(V)=\bigcup_{a\in V}\supp^\circ\widehat a=\bigcup_{a\in V}\kappa^{-1}(\CHECK a)=\kappa^{-1}\Bigl(\bigcup_{a\in V}\CHECK a\Bigr)
  =\kappa^{-1}(\mathcal U_{V,1})\,.\qedhere
  \]
\end{proof}

\subsubsection{The support and restriction maps}

In \cite[section~5.4]{ReSa16} to each quotient vector bundle $\mathcal A=(X,A,\kappa)$  was associated an adjoint pair $\sigma\dashv\gamma$:
the \emph{support map} $\sigma\colon\Sub A\to \Omega X$ and its right adjoint, the \emph{restriction map} $\gamma\colon\Omega X\to\Sub A$, defined by
\begin{equation}\label{eq:defnsupportrestriction}
  \sigma(V)=\bigcup_{a\in V}\supp^\circ\widehat a\quad\text{and}\quad
  \gamma(U)=\text{span}\bigl\{a\in A:\supp^\circ\widehat a\subset U\bigr\}\,.
\end{equation}
In particular, for any $a\in A$ we
have $\sigma\bigl(\linspan a\bigr)=\supp^\circ\widehat a$. 
The \emph{spectral kernel} of $\mathcal A$ is the restriction of $\gamma$ to $\Sigma\Omega X$, which we represent by $\kfrak$.
We say the quotient vector bundle $(X,A,\kappa)$ is \emph{spectral} if it has
the open support property and its spectral kernel is continuous.
We represent the category of spectral quotient vector bundles by $\QVBun^*_\Sigma$.

\subsection{Linearized locales}\label{sec:prelimlinloc}

A linearized locale \cite[section~5.1]{ReSa16} is a triple $(L,A,\gamma)$ where $L$ is a locale, $A$ is a topological vector space and $\gamma\colon L\to\Sub A$
is an inf-lattice homomorphism whose restriction to $\Sigma L$ is continuous. Given linearized locales $\mathfrak A=(L_A,A,\gamma_A)$ and
$\mathfrak B=(L_B,B,\gamma_B)$, 
a (contravariant) morphism $\mathfrak A\to\mathfrak B$ is a pair $(\underline f,\overline f)$ where $\underline f\colon L_A\to L_B$
is a map of locales and $\overline f\colon B\to A$ is a continuous linear map satisfying the relation
\begin{equation}\label{eq:morLinLoc}
\gamma_B\circ\underline f_*\leq\overline f^{-1}\circ\gamma_A\,.
\end{equation}
We represent the category of linearized locales by $\LLoc^*$.

From a linearized locale
$\mathfrak A=(L,A,\gamma)$ we obtain a spectral quotient vector bundle $\mathbf\Sigma\mathfrak A=(\Sigma L,A,\kfrak)$, where
$\kfrak$ is the restriction of $\gamma$ to $\Sigma L$, and from a spectral quotient vector bundle $\mathcal A=(X,A,\kappa)$
we obtain a linearized locale $\mathbf\Omega\mathcal A=(\Omega X,A,\gamma)$, where $\gamma$ is the restriction map.
These assignements extend to adjoint functors $\mathbf\Omega$ and $\mathbf\Sigma$
between the category of spectral quotient vector bundles and the category of linearized locales \cite[Theorems~5.15 and~5.19]{ReSa16}.

\section{Pseudo quotient vector bundles}\label{sec:3}

A \emph{pseudo quotient vector bundle} is a triple $(X,A,\kappa)$ where $X$ is a topological space, 
$A$ is a locally convex topological vector space
and $\kappa\colon X\to\Sub A$ is any map (which we don't require to be continuous). A morphism of pseudo quotient vector bundles
$(X,A,\kappa_A)\to(Y,B,\kappa_B)$ is a pair $(f_\flat,f^\sharp)$ where $f_\flat\colon X\to Y$ is continuous and $f^\sharp\colon B\to A$
is a linear map (not necessarily continuous) obeying condition~\eqref{eq:morQVBun}. 
We say the morphism is strict if it satisfies condition~\eqref{eq:strictmorQVBun}.
We represent by $\PsiQVBun^*$ the category of pseudo quotient vector bundles.

\begin{lemma}\label{QVBunfisostrict}
	A morphism of pseudo quotient vector bundles $(f_\flat,f^\sharp)\colon(X,A,\kappa_A)\to(Y,B,\kappa_B)$ 
	is an isomorphism if and only if it is strict,
	$f_\flat$ is a homeomorphism and $f^\sharp$ is a vector space isomorphism.
\end{lemma}
\begin{proof}
	A morphism of pseudo quotient vector bundles $f=(f_\flat,f^\sharp)$ 
	is an isomorphism if and only if $f_\flat$ and $f^\sharp$ are isomorphisms, and
	the pair $\bigl((f_\flat)^{-1},(f^\sharp)^{-1}\bigr)$ is a morphism of quotient vector bundles, in which case
	$f^{-1}=\bigl((f_\flat)^{-1},(f^\sharp)^{-1}\bigr)$. 
	The condition that $\bigl((f_\flat)^{-1},(f^\sharp)^{-1}\bigr)$ be a morphism of quotient vector bundles:
	\[
	a\in \kappa_A\bigl(f_\flat^{-1}(y)\bigr)\ \Rightarrow\ (f^\sharp)^{-1} (a)\in \kappa_B(y)\quad(y\in Y,\ a\in A)\,,
	\]
	is equivalent to
	\[
	f^\sharp (b)\in \kappa_A(x)\ \Rightarrow\ b\in \kappa_B\bigl(f_\flat(x)\bigr)\quad(x\in X,\ b\in B)
	\]
	which is equivalent to $(f_\flat,f^\sharp)$ being a strict morphism.\qedhere
\end{proof}

From a pseudo quotient vector bundle $\mathcal A=(X,A,\kappa)$ we can obtain a quotient
vector bundle $\mathcal A^{\mathrm {triv}}$ by replacing the topology of $A$ with the trivial (indiscreet) topology.

\begin{lemma}
  The assignments $\mathcal A\to\mathcal A^{\mathrm{triv}}$ 
  and $(f_\flat,f^\sharp)\to(f_\flat,f^\sharp)$ define a 
  fully faithful functor $\PsiQVBun^*\to\QVBun^*$ which is left adjoint to the inclusion functor $\QVBun^*\to\PsiQVBun^*$.
  These functors define an equivalence between the category $\PsiQVBun^*$ and the full subcategory of $\QVBun^*$ whose objects are the triples $(X,A,\kappa)$
  where $A$ has the trivial topology.
\end{lemma}
\begin{proof}
  The adjunction holds since, if $A^{\mathrm{triv}}$ has the trivial topology, any map $B\to A^{\mathrm{triv}}$ is continuous.
  To prove the equivalence of the categories observe that for any $\mathcal A=(X,A,\kappa)\in\PsiQVBun^*$ 
  we have a natural isomorphism $\mathcal A\cong\mathcal A^{\mathrm{triv}}$
  in $\PsiQVBun^*$ defined by the identity maps $X\to X$ and $A\to A$.\qedhere
\end{proof}

We say a pseudo quotient vector bundle $\mathcal A=(X,A,\kappa)$ is spectral if the quotient vector bundle $\mathcal A^{\mathrm{triv}}$ is spectral.
This is equivalent to saying that $\mathcal A^{\mathrm{triv}}$ has the open support property.
We represent the category of spectral pseudo quotient vector bundles by $\PsiQVBun^*_\Sigma$.
We define the support and restriction maps of $\mathcal A$ as the support and restriction maps
of the bundle $\mathcal A^{\mathrm{triv}}$ in equation~\eqref{eq:defnsupportrestriction}
by identifying $\Sub A$ with $\Sub A^{\mathrm{triv}}$.

\begin{lemma}\label{lemmaOpenSupportiffpseudo}
	A pseudo quotient vector bundle $\mathcal A=(X,A,\kappa)$ is spectral
	if and only if $\sigma(V)=\kappa^{-1}(\mathcal U_{V,1})$ for any $V\in\Sub A$.
\end{lemma}
\begin{proof}
	It follows immediately from Lemma~\ref{lemmaOpenSupportiff}.\qedhere
\end{proof}

\subsection{Pseudo linearized locales and adjunction}

A pseudo linearized locale is a triple $(L,A,\gamma)$ where $L$ is a locale, 
$A$ is a locally convex topological vector space
and $\gamma\colon L\to\Sub A$ is an inf-lattice homomorphism. 
To a pseudo linearized locale $\mathfrak A=(L,A,\gamma)$ we associate a linearized locale $\mathfrak A^{\mathrm{triv}}$ 
obtained by giving $A$ the trivial topology. A morphism of pseudo linearized locales $\mathfrak A\to\mathfrak B$
is a morphism $\mathfrak A^{\mathrm{triv}}\to\mathfrak B^{\mathrm{triv}}$ in $\LLoc^*$.
We let $\PsiLLoc^*$ denote the category of pseudo linearized locales.

Consider the adjunction $\boldsymbol\Omega\dashv\boldsymbol\Sigma$ from 
section~\ref{sec:prelimlinloc}. 

\begin{theorem}
The adjunction $\boldsymbol\Omega\dashv\boldsymbol\Sigma$ extends to an adjunction between the categories
$\PsiQVBun^*_\Sigma$ and $\PsiLLoc^*$.
\end{theorem}
\begin{proof}
The functors $\boldsymbol\Omega$ and $\boldsymbol\Sigma$ extend to functors
$\PsiQVBun^*\to\PsiLLoc^*$ and $\PsiLLoc^*\to\PsiQVBun^*$ respectively
in the obvious way: given a pseudo quotient vector bundle $(X,A,\kappa)$, we let
$\boldsymbol\Omega(X,A,\kappa)=(\Omega X,A,\gamma)$
and given a pseudo linearized locale $(L,A,\gamma)$ we let $\Sigma(L,A,\gamma)=(\Sigma L,A,\kfrak)$.
To define the functor on morphisms just recall that
a morphism $\mathcal A\to\mathcal B$ is a morphism $\mathcal A^{\text{triv}}\to\mathcal B^{\text{triv}}$
in $\QVBun^*$ and similarly for linearized locales.

	We represent by $\QVBun^*_{\mathrm{triv},\Sigma}$ and $\LLoc^*_{\mathrm{triv}}$ respectively the images of
	$\PsiQVBun^*_\Sigma$ and $\PsiLLoc^*$ under the functors
	$\mathcal A\mapsto\mathcal A^{\mathrm{triv}}$ and $\mathfrak A\mapsto\mathfrak A^{\mathrm{triv}}$.
The adjunction $\boldsymbol\Omega\dashv\boldsymbol\Sigma$ between the categories $\QVBun^*_\Sigma$ and $\LLoc^*$
restricts to an adjunction between $\QVBun^*_{\mathrm{triv},\Sigma}$ and $\LLoc^*_{\mathrm{triv}}$. 
Noting that $\boldsymbol\Omega(\mathcal A^{\mathrm{triv}})=(\boldsymbol\Omega\mathcal A)^{\mathrm{triv}}$
and $\boldsymbol\Sigma(\mathfrak A^{\mathrm{triv}})=(\boldsymbol\Sigma\mathfrak A)^{\mathrm{triv}}$
we obtain the commutative diagrams
\[
\xymatrix{
	\PsiQVBun^*_\Sigma              \ar@<1ex>[r]^-{\boldsymbol\Omega} \ar[d]^{\text{triv}}_{\cong} &
	\PsiLLoc_*                      \ar@<1ex>[l]^-{\Sigma} \ar[d]^{\text{triv}}_{\cong} \\
	\QVBun^*_{\mathrm{triv},\Sigma} \ar@<1ex>[r]^-{\boldsymbol\Omega}\ar@{}[r]|-{\perp} &
	\LLoc^*_{\mathrm{triv}}         \ar@<1ex>[l]^-{\Sigma} }
\]
which completes the proof.\qedhere
\end{proof}

\subsection{Coarse topologies on \boldmath $\Sub A$}

\subsubsection{The coarse open support (COS) topology}

Recall the definition of $\mathcal U_{V_0,V_1}\subset\Sub A$ in equation~\eqref{eq:mathcalU}.
We define the \emph{coarse open support topology}
on $\Sub A$ as the coarsest topology for which all sets $\CHECK a=\mathcal U_{{\linspan{a}},1}$ are open sets.
This coincides with the open support topology in \cite{ReSa16} when $A$ has the trivial topology.
Recall the definition of $\mathcal U_{V_0,V_1}\subset\Sub A$ in equation~\eqref{eq:mathcalU}.

\begin{lemma}\label{sigma=kappa-1circsobopensupp}\label{opensupporttopology}
  The space $\Sub A$ with the COS topology is a sober space
  and its soberification map is given by $\sob(V)=\mathcal U_{V,1}$.
\end{lemma}
\begin{proof}
  We first prove that $\sob(V)=\mathcal U_{V,1}$. From Lemma~\ref{lemma:bigcupCHECKa} the set $\mathcal U_{V,1}=\bigcup_{a\in V}\CHECK a$
  is open and $V\notin\mathcal U_{V,1}$
  so we must have $\mathcal U_{V,1}\subset\sob(V)$. Conversely, if $W\notin\mathcal U_{V,1}$ then
  $V\subset W$. It follows that any neighbourhood of $W$ is also a neighbourhood of $V$ so $W\notin\sob(V)$.

  Now we show that the map $\sob$ is an isomorphism.
  Any prime open set is the union of subbasic open sets so from Lemma~\ref{lemma:bigcupCHECKa}
  it must be of the form $\mathcal U_{V,1}=\sob(V)$ for some $V\in\Sub A$.
  This shows surjectivity. To show injectivity we just need to check that
  the space $\Sub A$ is $T_0$. If $V\neq W$ we may assume without loss of generality that there is some $a\in V\setminus W$.
  Then $W\in\CHECK a$ but $V\notin\CHECK a$. So $\sob$ is also injective, which completes the proof.\qedhere
\end{proof}

\begin{corollary}\label{cor:sigma=kappa-1circsobopensupp}
  Suppose $\Sub A$ has the coarse open support topology,
  let $\sob$ be the soberification map and let $i_\Sigma\colon\Sigma\Omega\Sub A\to\Omega\Sub A$ be the inclusion.
  Then a pseudo quotient vector bundle $(X,A,\kappa)$
  has the open support property if and only if its support map $\sigma\colon\Sub A\to\Omega X$ is given by
  $\sigma=\kappa^{-1}\circ i_\Sigma\circ\sob$.
\end{corollary}
\begin{proof}
  It follows immediately from Lemmas~\ref{lemmaOpenSupportiffpseudo} and~\ref{opensupporttopology}.\qedhere
\end{proof}

\subsubsection{The dual COS topology on \boldmath $\Max A$}

We call dual COS topology to the topology
$\mathcal T$ on $\Max A$ generated by the sets $\mathcal U_{0,V}$, with $V\in \Max A$.

\begin{lemma}\label{le:F0,Vinflatticehom}
  For each $V\in\Sub A$ let
  \[
  \mathcal F_{0,V}=\Sub A\setminus\mathcal U_{0,V}=\{W\in \Sub A:W\subset V\}
  \]
  Then for any family $\{V_\alpha\}$ in $\Sub A$ we have
  \begin{equation}
    \mathcal F_{0,\bigcap_\alpha V_\alpha}=\bigcap_\alpha\mathcal F_{0,V_\alpha}\,.
  \end{equation}
\end{lemma}
\begin{proof}
  It is enough to observe that for any $W\in\Sub A$, we have
  $W\subset\bigcap_\alpha V_\alpha$ if and only if $W\subset V_\alpha$ for any $\alpha$.\qedhere
\end{proof}

\begin{lemma}\label{lemma:sob=sob}
  The space $\Max A$ with the topology $\mathcal T$ is sober, with soberification map given by $\sob(V)=\mathcal U_{0,V}$.
\end{lemma}
\begin{proof}
	First we observe that any prime open set is of the form $\mathcal U_{0,V}$ for some $V\in\Max A$.
	This follows since any prime open set is the union of subbasic open sets so by Lemma~\ref{le:F0,Vinflatticehom}
	it must be of the form $\bigcup_\alpha\mathcal U_{0,V_\alpha}=\mathcal U_{0,\bigcap_\alpha V_\alpha}$.
	Now $\sob(V)=\Max A\setminus\overline{\{V\}}=\mathcal U_{0,W}$ for some
	$W\in \Max A$ so clearly we must have $\sob(V)=\mathcal U_{0,V}$. 
	To complete the proof we just need to observe that the map $\sob(V)$ is clearly injective.\qedhere
\end{proof}

In particular, the soberification map of the topology $\mathcal T$ coincides with the soberification map of the lower Vietoris topology (equation~\eqref{eq:soblV(V)=U_0,V}) so we will represent both of them by $\sob_{\Max A}$.

\subsection{Bundles with Hausdorff fibers}

We say a pseudo quotient vector bundle has Hausdorff fibers if its kernel map factors through $\Max A$.
We represent by $\PsiQVBun^{1*}$ the subcategory of $\PsiQVBun^*$
whose objects are the pseudo quotient vector bundles with Hausdorff fibers
and whose morphisms $(X,A,\kappa_A)\to(Y,B,\kappa_B)$ are the pairs $(f_\flat,f^\sharp)$ such that $(f^\sharp)^{-1}(V)\in\Max B$ for every $V\in\Max A$.
We write $\PsiQVBun_\Sigma^{1*}$ for the full subcategory of $\PsiQVBun^{1*}$ whose objects are the
spectral pseudo quotient vector bundles.

\begin{lemma}
The adjunction $\boldsymbol\Omega\dashv\boldsymbol\Sigma$ restricts to an adjunction between 
$\PsiQVBun_\Sigma^{1*}$ and $\PsiLLoc^{1*}$.
\end{lemma}
\begin{proof}
	We just need to observe that 
	the functors $\boldsymbol\Omega$ and $\boldsymbol\Sigma$ restrict to functors between the categories
	$\PsiQVBun^{1*}$ and $\PsiLLoc^{1*}$ and the unit and counit maps $\mathcal A\to\boldsymbol{\Sigma\Omega}\mathcal A$
	and $\boldsymbol{\Omega\Sigma}\mathfrak A\to\mathfrak A$ are morphisms in $\PsiQVBun^{1*}$ and $\PsiLLoc^{1*}$ respectively.\qedhere
\end{proof}

We say a pseudo quotient vector bundle $(X,A,\kappa)$ is cospectral if the map $\kappa\colon X\to\Max A$ is continuous
when $\Max A$ is given the dual COS topology $\Tt$.
We represent by $\PsiQVBun_{\mathcal T}^{1*}$ the full subcategory of
$\PsiQVBun^{1*}$ whose objects are the cospectral pseudo quotient vector bundles.

\begin{lemma}\label{lemma:k(y)subsetk(x)}
  Let $\mathcal A=(X,A,\kappa)\in\PsiQVBun^{1*}_{\mathcal T}$ and let $x\in X$. Then for any $y\in\overline{\{x\}}$ we have $\kappa(y)\subset\kappa(x)$.
\end{lemma}
\begin{proof}
  Suppose $\Max A$ has the dual COS topology. 
  If $y\in\overline{\{x\}}$ then $\kappa(y)\in\kappa\bigl(\overline{\{x\}}\bigr)\subset\overline{\{\kappa(x)\}}$
  and by Lemma~\ref{lemma:sob=sob}, $\overline{\{\kappa(x)\}}=\{W\in\Max A:W\subset\kappa(x)\}$ so $\kappa(y)\subset\kappa(x)$.\qedhere
\end{proof}

\section{Quotient vector bundles via colocales}\label{sec:4}

\subsection{Linearized colocales}\label{sec:deflincoloc}

\subsubsection{Colocales}

We say a lattice $L^\complement$ is a \emph{colocale} if $(L^\complement)^\opp$ is a locale. 
A map of colocales $f\colon L^\complement\to M^\complement$ is a map of locales $f\colon (L^\complement)^\opp\to (M^\complement)^\opp$, that is, an inf-lattice homomorphism
in the opposite direcion $f^*\colon M^\complement\to L^\complement$ which also preserves finite joins. As
in the case of locales, we will refer to $f^*$ as the inverse image homomorphism
and to its left adjoint $f_*\colon L^\complement\to M^\complement$ as the direct image homomorphism.
The main example of a colocale is the collection
$\clsets X$ of closed sets on a topological space $X$,
which we identify with $(\Omega X)^\opp$
through the isomorphism $\complement\colon \clsets X\xrightarrow\cong(\Omega X)^\opp$ defined by $\complement(C)=X\setminus C$.

Given a colocale $L^\complement$,
let $\Iota L^\complement$ be the set of elements $c\in L^\complement$ such that
\[
c\leq a\vee b\ \Rightarrow\ c\leq a\text{ or }c\leq b\,.
\]
If $\clsets X$ is the collection of closed subsets of $X$ then $\Iota\clsets X$ is the collection of irreducible closed sets.
Notice that as sets we have $\Iota L^\complement=\Sigma (L^\complement)^\opp$.
We define the topology of $\Iota L^\complement$ by setting $\Omega\Iota L^\complement=\Omega\Sigma (L^\complement)^\opp$. 
If for each $a\in L^\complement$ we let
\[
C_a=\{c\in \Iota L^\complement:c\leq a\}\,,
\]
then $\{C_a\}_{a\in L^\complement}$ is the collection of closed sets in $\Iota L^\complement$.

The adjunction between topological spaces and locales has its counterpart for colocales. In particular,
given a colocale $L^\complement$, to the spatialization map
$\spat_{(L^\complement)^\opp}\colon \Omega\Sigma (L^\complement)^\opp\to (L^\complement)^\opp$
there corresponds
the map of colocales:
\begin{equation}\label{eq:spatc}
  \spat_{L^\complement}^\complement\colon\clsets\Iota {L^\complement}\to {L^\complement}
\end{equation}
whose inverse image homomorphism is given by:
\begin{equation}\label{eq:spatc*=Ca}
  (\spat_{L^\complement}^\complement)^*(a)=C_a\,.
\end{equation}
Now consider the lattice $\Max A$ with the lower Vietoris topology.
Let $\sob_{\Max A}^\complement=\complement\circ\sob_{\Max A}\colon \Max A\to\Iota\clsets\Max A$.
Then, from equation~\eqref{eq:soblV(V)=U_0,V} we get
\begin{equation}\label{eq:soblVC(V)=F_0,V}
\sob_{\Max A}^\complement(V)=\mathcal F_{0,V}\,.
\end{equation}

\begin{lemma}\label{sobmeetsjoins}\label{sigmaUisinflatticehom}\label{isigmasuplatticehom}
Let $A$ be a topological vector space. The inclusion map $i_\Iota\colon\Iota\clsets\Max A\to\clsets\Max A$ 
and the map $i_\Iota\circ\sob_{\Max A}^\complement\colon\Max A\to\clsets\Max A$ are inf-lattice homomorphisms.
\end{lemma}
\begin{proof}
From Lemma~\ref{le:F0,Vinflatticehom} and the identity
$\sob_{\Max A}^\complement(V)=\mathcal F_{0,V}$ we see that the map 
$i_\Iota\circ\sob_{\Max A}^\complement\colon\Max A\to\clsets\Max A$ is an inf-lattice homomorphism.
Since $\Max A$ is sober, $\sob_{\Max A}^\complement$ is a monotone homeomorphism so 
the map $i_\Iota\colon\Iota\clsets\Max A\to\clsets\Max A$ is also an inf-lattice homomorphism.\qedhere
\end{proof}

\subsubsection{Linearized colocales}

A \emph{pseudo linearized colocale} is a triple $\mathfrak A^\complement=(L^\complement,A,\gamma^\complement)$ where
$L^\complement$ is a colocale, $A$ is a locally convex topological vector space 
and $\gamma^\complement\colon L^\complement\to \Max A$ is a sup-lattice homomorphism.
Denote by $\kfrak^\complement$ the restriction of $\gamma^\complement$ to $\Iota L^\complement$
and by $\sigma^\complement\colon\Max A\to L^\complement$ the inf-lattice homomorphism
right adjoint to $\sigma^\complement$.

\begin{lemma}\label{lemma:kfrak-1U0,Vopen}
  The map $\kfrak^\complement\colon \Iota L^\complement\to\Max A$ is continuous
  if $\Max A$ has the dual COS topology.
\end{lemma}
\begin{proof}
The dual COS topology is generated by the collection $\{\mathcal U_{0,V}\}_{V\in\Max A}$.
Given $V\in \Max A$ we have
\begin{align*}
  (\kfrak^\complement)^{-1}(\mathcal F_{0,V})&=\bigl\{a\in\Iota L^\complement:\gamma^\complement(a)\in\mathcal F_{0,V}\bigr\}\\
  &=\bigl\{a\in\Iota L^\complement:\gamma^\complement(a)\subset V\bigr\}\\
  &=\bigl\{a\in\Iota L^\complement:a\leq\sigma^\complement(V)\bigr\}\\
  &=(\spat_{L^\complement}^\complement)^{*}\bigl(\sigma^\complement(V)\bigr)
\end{align*}
so $(\kfrak^\complement)^{-1}(\mathcal F_{0,V})$ is closed for any $V\in \Max A$ which shows that $\kfrak^\complement$ is continuous.\qedhere
\end{proof}

By Lemma~\ref{lemma:kfrak-1U0,Vopen},
a pseudo linearized colocale $\mathfrak A^\complement=( L^\complement,A,\gamma^\complement)$ gives rise
to a pseudo quotient vector bundle 
$\boldsymbol\Iota\mathfrak A^\complement\in\PsiQVBun^{1*}_{\Tt}$, namely
$\boldsymbol\Iota\mathfrak A^\complement=(\Iota L^\complement,A,\kfrak^\complement)$. 

We call $\mathfrak A^\complement$
a \emph{linearized colocale} if the map $\kfrak^\complement$ is continuous
when $\Max A$ is given the lower Vietoris topology. In that case
$\boldsymbol\Iota\mathfrak A^\complement$ is a quotient vector bundle.

\subsection{Cosupport and corestriction maps}

Let $\mathcal A=(X,A,\kappa)\in\PsiQVBun^{1*}_{\mathcal T}$.
We call the map $\sigma^\complement\colon \Max A\to\clsets X$ defined by
\begin{equation}\label{eq:defcosupport}
  \sigma^\complement(V)=\kappa^{-1}(\mathcal F_{0,V})
\end{equation}
the \emph{cosupport map} of $\mathcal A$ (compare with Lemma~\ref{lemmaOpenSupportiff}). %
Using the notation of Lemma~\ref{sigmaUisinflatticehom} and equation~\eqref{eq:soblVC(V)=F_0,V} we have:
\begin{equation}\label{eq:defcosupportforMaxA}
  \sigma^\complement=\kappa^{-1}\circ i_\Iota\circ\sob_{\Max A}^\complement
\end{equation}
(compare with Corollary~\ref{cor:sigma=kappa-1circsobopensupp}).

Since the map $\kappa^{-1}$ is an inf-lattice homomorphism, 
Lemma~\ref{le:F0,Vinflatticehom} immediately implies:

\begin{lemma}\label{sigmameetsjoins}
The map $\sigma^\complement$ is an inf-lattice homomorphism.
\end{lemma}

\begin{definition}\label{def:corestriction}
  Given a pseudo quotient vector bundle $\mathcal A=(X,A,\kappa)\in\PsiQVBun^{1*}_{\mathcal T}$
  we call \emph{corestriction map} of $\mathcal A$ to the map $\gamma^\complement\colon\clsets X\to\Max A$ left adjoint to $\sigma^\complement$.
  We represent the restriction of $\gamma^\complement$ to $\Iota\clsets X$ by $\kfrak^\complement$
  and we represent by $\boldsymbol\clsets\mathcal A$ the
  pseudo linearized colocale $\boldsymbol\clsets\mathcal A=(\Iota\clsets X,A,\gamma^\complement)$.
\end{definition}

\begin{lemma}\label{lemmagamma=bigvee}
Given a pseudo quotient vector bundle $(X,A,\kappa)\in \PsiQVBun^{1*}_{\mathcal T}$, for any closed set $C\in\clsets X$ we have:
\[
\gamma^\complement(C) %
=\bigvee_{x\in C}\kappa(x)
\]
\end{lemma}
\begin{proof}
Since $\sigma^\complement(V)=\kappa^{-1}(\mathcal F_{0,V})$, we have
\[
C\subset\sigma^\complement(V)\Leftrightarrow \kappa(C)\subset\mathcal F_{0,V}\Leftrightarrow \underset{x\in C}{\forall}\kappa(x)\subset V
\Leftrightarrow \bigvee_{x\in C}\kappa(x)\subset V
\]
which shows that the assignement $C\mapsto \bigvee_{x\in C}\kappa(x)$ defines a map left adjoint to $\sigma^\complement$,
finishing the proof.\qedhere
\end{proof}

\begin{lemma}\label{kfrakcircsob=kappa}
  Let $(X,A,\kappa)\in\PsiQVBun^{1*}_{\mathcal T}$. Then $\kappa=\kfrak^\complement\circ\sob_X^\complement$.
\end{lemma}
\begin{proof}
  It follows from Lemmas~\ref{lemmagamma=bigvee} and~\ref{lemma:k(y)subsetk(x)}:
  \[
  \kfrak^\complement\circ\sob_X^\complement(x)=\bigvee_{y\in\overline{\{x\}}}\kappa(y)=\kappa(x)\,.\qedhere
  \]
\end{proof}

 Recall (equation~\eqref{eq:f!c}) that a continuous map $f\colon X\to Y$ induces a map
 $f_!^\complement\colon\Iota\clsets X\to\Iota\clsets Y$.
 
\begin{lemma}\label{kfrakcircsob=kappa2}
  Let $\mathcal A=(X,A,\kappa)$ be a quotient vector bundle with Hausdorff fibers.
  Give $\Max A$ the lower Vietoris topology. Then the following diagram is commutative:
\[
\xymatrix{
\Max A\ar[rr]^-{\sob_{\Max A}^\complement}&&
\Iota\clsets\Max A\\
X\ar[u]^-{\kappa}\ar[rr]_-{\sob_X^\complement}&&
\Iota\clsets X\ar[u]_-{\kappa_!^\complement}\ar[ull]^-{\kfrak^\complement}}
\]
\end{lemma}
\begin{proof}
  By Lemma~\ref{kfrakcircsob=kappa}
  we only need to check the commutativity of the upper triangle. 
  By \cite[Theorem~6.16]{ReSa16}, the map $\sob_{\Max A}^\complement$ is a homeomorphism.
  Let $r_\Iota$ be the left adjoint of the inclusion
  $i_\Iota\colon\Iota\clsets\Max A\to\clsets\Max A$.
  Then $r_\Iota$
  is the identity on $\Iota\clsets\Max A$.
  Taking left adjoints of the identity $\sigma^\complement=k^{-1}\circ i_\Iota\circ\sob_{\Max A}^\complement$  we get
  $\gamma^\complement=\bigl(\sob_{\Max A}^\complement\bigr)^{-1}\circ r_\Iota\circ k_!^\complement$.
  Restricting to $\Iota\clsets X$, from
  $k_!^\complement(\Iota\clsets X)\subset \Iota\clsets\Max A$
  we get $\kfrak^\complement=\bigl(\sob_{\Max A}^\complement\bigr)^{-1}\circ k_!^\complement$
  which shows the commutativity of the upper triangle.\qedhere
\end{proof}

\begin{remark*}
	Giving $\Max A$ the dual COS topology, Lemma~\ref{kfrakcircsob=kappa2} also holds for pseudo quotient vector bundles. In particular
	the map $\kappa_!^\complement$ is the same when $\Max A$ is given either the lower Vietoris topology or the dual COS topology.
\end{remark*}

\begin{corollary}\label{kfrakiscontinuous}
  Let $\mathcal A=(X,A,\kappa)$ be a quotient vector bundle with Hausdorff fibers.
  Then the map $\kfrak^\complement\colon I\clsets X\to\Max A$ is continuous so
  $\boldsymbol{\clsets}\mathcal A$ is a linearized colocale.
\end{corollary}
\begin{proof}
  Since $\sob_{\Max A}^\complement$ is a homeomorphism and $\kfrak^\complement=\bigl(\sob_{\Max A}^\complement\bigr)^{-1}\circ \kappa_!^\complement$,
  the result follows from the continuity of $\kappa_!^\complement$.\qedhere
\end{proof}

\subsection{Soberification}

Let $\mathcal A=(X,A,\kappa)$ be a spectral quotient vector bundle. 
Since $\mathcal A$ has the open support property then we have the identity:
\begin{equation}\label{eq:kfrakcircsob=kappa}
\kfrak\circ\sob_X=\kappa\,.
\end{equation}
We have a soberification map $\sob_{\mathcal A}\colon\mathcal A\to\boldsymbol{\Sigma\Omega}\mathcal A$
defined by $\sob_{\mathcal A}=(\sob_X,\ident)$ \cite[section~5.4]{ReSa16}.

\begin{lemma}\label{kappaconstantonC}
	If $\mathcal A=(X,A,\kappa)$ is a spectral quotient vector bundle and $C\subset X$ is an irreducible closed set, then the kernel map $\kappa$ is constant on $C$, equal to $\kfrak(X\setminus C)$.
\end{lemma}
\begin{proof}
	It was shown in \cite[Theorem~5.13]{ReSa16} that, if $\kfrak$ is continuous and $P,Q\in\Sigma\Omega X$ are such that $P\subset Q$ then $\kfrak(P)=\kfrak(Q)$.
	Let $x\in C$ and let $P=X\setminus C\in\Sigma\Omega X$. Then $P\subset\sob_X(x)$ so, by equation~\eqref{eq:kfrakcircsob=kappa},
	$\kappa(x)=\kfrak\bigl(\sob_X(x)\bigr)=\kfrak(P)$ as claimed.\qedhere
\end{proof}

\begin{corollary}\label{cor:kfrak=kfrak}
	Let $\mathcal A=(X,A,\kappa)$ be a spectral quotient vector bundle. Then
	$\kfrak^\complement(C)=\kfrak(X\setminus C)$ for every irreducible closed set $C\subset X$.
\end{corollary}
\begin{proof}
	Given any irreducible closed set $C\in\Iota\clsets X$,
	by Lemmas~\ref{lemmagamma=bigvee} and~\ref{kappaconstantonC}
	we have:
	\[
	\kfrak^\complement(C)=\bigvee_{x\in C}\kappa(x)=\kfrak(X\setminus C)\,.\qedhere
	\]
\end{proof}

\begin{lemma}\label{lemma:sobAisstrictmor}
	Let $\mathcal A=(X,A,\kappa)\in\PsiQVBun^{1*}_\Tt$.
	Then the pair $\sob_{\mathcal A}^\complement=(\sob_X^\complement,\ident)$
	is a strict morphism of pseudo quotient vector bundles $\mathcal A\to\boldsymbol\Iota\boldsymbol{\clsets}\mathcal A$
	which is an isomorphism if and only if $X$ is sober.
\end{lemma}
\begin{proof}
	Write $\boldsymbol\Iota\boldsymbol{\clsets}\mathcal A=(\Iota\clsets X,A,\kfrak^\complement)$.
	From the identity $\kappa=\kfrak^\complement\circ\sob_X^\complement$ (Lemma~\ref{kfrakcircsob=kappa}) we get that,
	for any $x\in X$ and $a\in A$:
	\[
	a\in \kfrak^\complement\bigl(\sob_X(x)\bigr)\ \Leftrightarrow\ a\in\kappa(x)
	\]
	so $\sob_{\mathcal A}^\complement$ is a strict morphism of pseudo quotient vector bundles, which,
	by Lemma~\ref{QVBunfisostrict},
	is an isomorphism if and only if $\sob_X^\complement$ is a homeomorphism.\qedhere
\end{proof}

\begin{theorem}\label{T4.12}
	If $\mathcal A$ is a spectral pseudo quotient vector bundle then the map
	$(\complement,\ident)\colon\boldsymbol\Iota\boldsymbol{\clsets}\mathcal A\to\boldsymbol\Sigma\boldsymbol\Omega\mathcal A$
	is an isomorphism of pseudo quotient vector bundles and the following diagram commutes:
	\[
	\xymatrix{
		\boldsymbol\Iota\boldsymbol{\clsets}\mathcal A\ar[rr]^-{(\complement,\ident)}&&
		\boldsymbol\Sigma\boldsymbol\Omega\mathcal A \\
		&\mathcal A\ar[lu]^-{\sob^\complement_{\mathcal A}}\ar[ru]_-{\sob_{\mathcal A}}}
	\]
\end{theorem}
\begin{proof}
	To show that the pair $(\complement,\ident)$ is an isomorphism we apply Lemma~\ref{QVBunfisostrict}.
	The map $\complement\colon\Iota\clsets X\to\Sigma\Omega X$ is a homeomorphism and
	given any irreducible closed set $C\in\Iota\clsets X$, 
	by Lemma~\ref{cor:kfrak=kfrak} we have
	$\kfrak^\complement(C)=\kfrak\bigl(\complement(C)\bigr)$
	so the pair $(\complement,\ident)$ is a strict morphism of quotient vector bundles.
	The commutativity of the diagram is immediate.\qedhere
\end{proof}

\subsection{The universal example}

Suppose $\Max A$ has the lower Vietoris topology.
The identity map $\ident\colon\Max A\to \Max A$ induces a quotient vector bundle $\mathcal A_U=(\Max A,A,\ident)$ which we call
the universal quotient vector bundle \cite[section~5]{ReSa17}. Denote by $\boldsymbol\clsets\mathcal A_U=(\clsets\Max A,A,\gamma_U^\complement)$
the associated linearized colocale.

\begin{lemma}\label{gamma=gammaUcirck!}
  Let $\mathcal A=(X,A,\kappa)$ be a pseudo
  quotient vector bundle. Then its corestriction map satisfies
  $\gamma^\complement=\gamma_U^\complement\circ k_!^\complement$.
\end{lemma}
\begin{proof}
  From Definition~\ref{def:corestriction}, the corestriction map $\gamma_U^\complement$ is left adjoint to 
  $\sigma_U^\complement=\ident^{-1}\circ i_\Iota\circ \sob_{\Max A}^\complement=i_\Iota\circ\sob_{\Max A}^\complement$
  and $\gamma^\complement$ is left adjoint to  $\sigma^\complement=\kappa^{-1}\circ i_\Iota\circ \sob_{\Max A}^\complement=\kappa^{-1}\circ\sigma_U^\complement$.
  Taking left adjoints of the identiy $\sigma^\complement=\kappa^{-1}\circ\sigma_U^\complement$ we get $\gamma^\complement=\gamma_U^\complement\circ k_!^\complement$.\qedhere
\end{proof}

\begin{lemma}\label{kfrakkU=sob-1}
  The restriction of $\gamma_U^\complement$ to $I\clsets\Max A$ equals $(\sob_{\Max A}^\complement)^{-1}$.
\end{lemma}
\begin{proof}
  It follows immediately from Lemma~\ref{kfrakcircsob=kappa} applied to the quotient vector bundle $(\Max A,A,\ident)$.\qedhere
\end{proof}

Given a pseudo linearized colocale $\mathfrak A^\complement=(L^\complement,A,\gamma^\complement)$, 
we say a map of colocales $\underline\kappa\colon L^\complement\to\clsets\Max A$ 
is a \emph{classifying map} for $\mathfrak A^\complement$
if $\gamma^\complement=\gamma_U^\complement\circ\underline\kappa_*$.

\begin{lemma}\label{lemma:uniqueclassifmap}
  Let $\mathcal A=(X,A,\kappa)$ be a quotient vector bundle. 
  Then the linearized colocale $\boldsymbol\clsets\mathcal A$ has a unique classifying map, namely
  $\kappa_!^\complement\colon\clsets X\to\clsets\Max A$.
\end{lemma}
\begin{proof}
  From Lemma~\ref{gamma=gammaUcirck!} it follows immediately
  that $\kappa_!^\complement$ is a classifying map.
  To prove uniqueness
  let $\underline\kappa$ be a classifying map for $\boldsymbol\clsets\mathcal A$. Restricting the identity $\gamma^\complement=\gamma_U^\complement\circ\underline\kappa_*$
  to $\Iota\clsets X$ and applying Lemma~\ref{kfrakkU=sob-1} we get $\underline\kappa_*=\sob_{\Max A}^\complement\circ\kfrak^\complement$,
  where $\kfrak^\complement$ is the restriction of $\gamma^\complement$ to $\Iota\clsets X$.
  Now for any $C\in\clsets X$ we have
  \begin{align*}
    \underline\kappa_*(C)&=\underline\kappa_*\Bigl(\bigvee_{x\in C}\sob_{X}^\complement(x)\Bigr)\\
    &=\bigvee_{x\in C}\underline\kappa_*\circ\sob_{X}^\complement(x)\\
    &=\bigvee_{x\in C}\sob_{\Max A}^\complement\circ\kfrak^\complement\circ\sob_{X}^\complement(x)
  \end{align*}
  which shows uniqueness.\qedhere
\end{proof}

\begin{corollary}
Let $\mathfrak A^\complement=(L^\complement,A,\gamma^\complement)$ be a spatial linearized colocale. 
Then the classifying map $\underline\kappa$ exists and is unique.
\end{corollary}
\begin{proof}
  It follows from Lemma~\ref{lemma:uniqueclassifmap} applied to the quotient vector bundle $\mathcal A=\Iota\mathfrak A^\complement$.\qedhere
\end{proof}

\section{Adjunctions}\label{sec:5}

\subsection{Covariant and contravariant morphisms}

\subsubsection{Linearized colocales}

Given pseudo linearized colocales 
$\mathfrak A^\complement=( L^\complement_A,A,\gamma_A^\complement)$ and
$\mathfrak B^\complement=( L^\complement_B,B,\gamma_B^\complement)$, a contravariant
morphism $\mathfrak f\colon\mathfrak A^\complement\to\mathfrak B^\complement$ is a pair
$f=(\underline f,\overline f)$ where $\overline f\colon B\to A$ is a linear map such that $\overline f{}^{-1}(V)\in \Max B$ for any $V\in \Max A$
and $\underline f\colon L^\complement_A\to L^\complement_B$ is a map of colocales, satisfying the relation %
\begin{equation}\label{eq:morofcolocales}
\Max\overline f\circ\gamma_B^\complement\circ \underline f_*\leq \gamma_A^\complement\,,
\end{equation}
where $\underline f_*$ is the direct image homomorphism of the map $\underline f$.
Composition is given by
$(\underline f,\overline f)\circ(\underline g,\overline g)=(\underline f\circ\underline g,\overline g\circ\overline f)$.
It is straightforward to check that composition preserves relation~\eqref{eq:morofcolocales}.
We represent by $\PsiLinCoLoc^{1*}$\ the category whose objects are the
pseudo linearized colocales and whose morphisms are the contravariant morphisms. 
We represent by $\LinCoLoc^{1*}$ the subcategory whose objects are the linearized colocales
and whose morphisms are the pairs $(\underline f,\overline f)$ where $\overline f$ is continuous.

Given linearized colocales $\mathfrak A^\complement=(L^\complement_A,A,\gamma_A^\complement)$
and $\mathfrak B^\complement=(L^\complement_B,B,\gamma_B^\complement)$,
a covariant morphism $\mathfrak A^\complement\to\mathfrak B^\complement$ is a pair $(\underline f,\overline f)$
where $\overline f\colon A\to B$ is a linear map such that $\overline f{}^{-1}(V)\in \Max A$ for any $V\in \Max B$ and
$\underline f\colon L^\complement_A\to L^\complement_B$ is a map of colocales satisfying the relation
\[
\Max\overline f\circ\gamma_A^\complement\leq\gamma_B^\complement\circ\underline f_*\,.
\]
Composition is given by
$(\underline f,\overline f)\circ(\underline g,\overline g)=(\underline f\circ\underline g,\overline f\circ\overline g)$.
We represent by $\PsiLinCoLoc_*$\ the category of pseudo linearized colocales with covariant morphisms
and we represent by  $\LinCoLoc_*$ the subcategory whose objects are the linearized colocales
and whose morphisms are the pairs $(\underline f,\overline f)$ where $\overline f$ is continuous.

\subsubsection{Quotient vector bundles}

A contravariant morphism of pseudo quotient vector bundles is a 
morphisms in $\PsiQVBun^{1*}$.
Given pseudo quotient vector bundles $\mathcal A=(X,A,\kappa_A)$ and $\mathcal B=(Y,B,\kappa_B)$,
a covariant morphism $\mathcal A\to\mathcal B$ is a pair $(f_\flat,f^\sharp)$
where $f_\flat\colon X\to Y$ is continuous, $f^\sharp\colon A\to B$ is a linear map
satisfying $(f^\sharp)^{-1}(V)\in\Max A$ for all $V\in\Max B$ and such that
for all $x\in X$ we have
\[
f^\sharp\bigl(\kappa_A(x)\bigr)\subset\kappa_B\bigl(f_\flat(x)\bigr)\,.
\]

Note that a covariant morphism induces, for each $x\in X$, linear maps $A/\kappa_A(x)\to B/\kappa_B\bigl(f_\flat(x)\bigr)$
and hence it defines acovariant  morphism between the associated linear bundles $E_A\to E_B$.
We represent by $\PsiQVBun^1_*$ the category whose objects are the pseudo quotient vector bundles 
with Hausdorff fibers and whose morphisms
are the covariant morphisms with composition given by 
$(f_\flat,f^\sharp)\circ(g_\flat,g^\sharp)=(f_\flat\circ g_\flat,f^\sharp\circ g^\sharp)$.
We represent by $\PsiQVBun^1_{\Sigma*}$ and $\PsiQVBun^1{\Tt*}$ the full subcategories
whose objects are respectively the spectral and the cospectral pseudo quotient vector bundles.
We represent by $\QVBun^1_*$ the subcategory whose objects are
the quotient vector bundles and whose morphisms are the pairs
$(f_\flat,f^\sharp)$ where $f^\sharp$ is continuous. 

\subsubsection{Linearized locales}

A contravariant morphism of pseudo linearized locales is a morphism in $\PsiLLoc^{1*}$.
Given pseudo linearized locales $\mathfrak A=(L_A,A,\gamma_A)$
and $\mathfrak B=(L_B,B,\gamma_B)$,
a covariant morphism $\mathfrak A\to\mathfrak B$ is a pair $(\underline f,\overline f)$
where $\overline f\colon A\to B$ is a linear map such that $\overline f{}^{-1}(V)\in \Max A$ for any $V\in \Max B$
and $\underline f\colon L_A\to L_B$ is a map of locales, satisfying the relation
\[
\gamma_A\leq\overline f^{-1}\circ\gamma_B\circ \underline f_*\,.
\]
We represent by $\PsiLLoc^1_*$ the category whose objects are the linearized
locales with covariant morphisms and we represent by $\LLoc_*$ the subcategory whose objects are the linearized locales
and whose morphisms are the pairs $(\underline f,\overline f)$ where $\overline f$ is continuous.

The following table summarizes the conditions satisfied by the morphisms of quotient vector bundles and linearized locales and colocales:
\[
\begin{array}{|l|l|l|}\hline
  \text{Category}    & \text{Morphisms} & \text{Condition} \\ \hline
              & (X,A,\kappa_A)\to(X,B,\kappa_B) & \\
  \raisebox{1.5ex}[0pt]{$\QVBun^{1*}$}      & f_\flat\colon X\to Y,\quad f^\sharp\colon B\to A &
  \raisebox{1.5ex}[0pt]{$b\in\kappa_B\bigl(f_\flat(x)\bigr)\Rightarrow f^\sharp(b)\in\kappa_A(x)$} \\ \hline
              & (X,A,\kappa_A)\to(X,B,\kappa_B) & \\
  \raisebox{1.5ex}[0pt]{$\QVBun^1_*$}      & f_\flat\colon X\to Y,\quad f^\sharp\colon A\to B &
  \raisebox{1.5ex}[0pt]{$a\in\kappa_A(x)\Rightarrow f^\sharp(a)\in\kappa_B\bigl(f_\flat(x)\bigr)$} \\ \hline
              & (L_A,A,\gamma_A)\to(L_B,B,\gamma_B) & \\
  \raisebox{1.5ex}[0pt]{$\LLoc^{1*}$}       & \underline f\colon L_A\to L_B,\quad \overline f\colon B\to A &  %
  \raisebox{1.5ex}[0pt]{$\gamma_B\circ\underline f{}_*\leq\overline f{}^{-1}\circ\gamma_A$} \\ \hline
              & (L_A,A,\gamma_A)\to(L_B,B,\gamma_B) & \\
  \raisebox{1.5ex}[0pt]{$\LLoc^1_*$}       & \underline f\colon L_A\to L_B,\quad \overline f\colon A\to B &
  \raisebox{1.5ex}[0pt]{$\gamma_A\leq\overline f{}^{-1}\circ\gamma_B\circ\underline f{}_*$} \\ \hline
              & (L_A^\complement,A,\gamma_A^\complement)\to(L_B^\complement,B,\gamma_B^\complement) & \\
  \raisebox{1.5ex}[0pt]{$\LinCoLoc^{1*}$}   & \underline f\colon L_A^\complement\to L_B^\complement,\quad \overline f\colon B\to A& 
  \raisebox{1.5ex}[0pt]{$\Max\overline f\circ\gamma_B^\complement\circ\underline f{}_*\leq\gamma_A^\complement$} \\ \hline
              & (L_A^\complement,A,\gamma_A^\complement)\to(L_B^\complement,B,\gamma_B^\complement) & \\
  \raisebox{1.5ex}[0pt]{$\LinCoLoc^1_*$}   & \underline f\colon L_A^\complement\to L_B^\complement,\quad \overline f\colon A\to B &
  \raisebox{1.5ex}[0pt]{$\Max\overline f\circ\gamma_A^\complement\leq\gamma_B^\complement\circ\underline f{}_*$} \\ \hline
\end{array}
\]

\subsubsection{\boldmath $\ident$-morphisms}

An $\ident$-morphism of pseudo quotient vector bundles $(X_1,A,\kappa_1)\to(X_2,A,\kappa_2)$ is a 
continuous map $f\colon X_1\to X_2$ such that for all $x\in X_1$ we have
\[
\kappa_1(x)=\kappa_2\bigl(f(x)\bigr)\,.
\]
This condition is equivalent to the pair $(f,\ident)$ being both a contravariant and a covariant morphism of pseudo quotient vector bundles.
We say $f$ is an $\ident$-isomorphism if $f$ is a homeomorphism.

Given pseudo linearized locales $\mathfrak A_1=(L_1,A,\gamma_1)$ and $\mathfrak A_2=(L_2,A,\gamma_2)$
and $\ident$-morphism $\mathfrak A_1\to\mathfrak A_2$ is a map of locales $\underline f\colon L_1\to L_2$ such that
\[
\gamma_1=\gamma_2\circ{\underline f}_*
\]
This condition is equivalent to the pair $(f,\ident)$ being both a contravariant and a covariant morphism of pseudo linearized locales.
We say $\underline f$ is an $\ident$-isomorphism if $\underline f$ is an isomorphism of locales.

Given pseudo linearized colocales 
$\mathfrak A_1^\complement=(L_1^\complement,A,\gamma_1^\complement)$ and 
$\mathfrak A_2^\complement=(L_2^\complement,A,\gamma_2^\complement)$
and $\ident$-morphism $\mathfrak A_1^\complement\to\mathfrak A_2^\complement$ is a map of colocales 
$\underline f\colon L_1^\complement\to L_2^\complement$ such that
\[
\gamma_1^\complement=\gamma_2^\complement\circ{\underline f}_*
\]
This condition is equivalent to the pair $(f,\ident)$ being both a contravariant and a covariant morphism of pseudo linearized colocales.
We say $\underline f$ is an $\ident$-isomorphism if $\underline f$ is an isomorphism of colocales.

\begin{lemma}\label{morlincolocisostrict}\hfill
	\begin{enumerate}
		\item An $\ident$-morphism of pseudo quotient vector bundles
		$f\colon(X_1,A,\kappa_1)\to(X_2,A,\kappa_2)$ is an $\ident$-isomorphism if and only if the pair $(f,\ident)$
		is both a covariant and a contravariant isomorphism of pseudo quotient vector bundles.
		\item An $\ident$-morphism of pseudo linearized locales
		$\underline f\colon( L_1,A,\gamma_1)\to( L_2,A,\gamma_2)$ is
		and $\ident$-isomorphism if and only if the pair $(\underline f,\ident)$ is both a covariant and a contravariant isomorphism of
		pseudo linearized locales.
	\item An $\ident$-morphism of pseudo linearized colocales
	$\underline f\colon( L^\complement_1,A,\gamma_1^\complement)\to( L^\complement_2,A,\gamma_2^\complement)$ is
	and $\ident$-isomorphism if and only if the pair $(\underline f,\ident)$ is both a covariant and a contravariant isomorphism of
	pseudo linearized colocales.
	\end{enumerate}
\end{lemma}
\begin{proof}
	The proofs are similar in the three cases so we will prove only (3).
	The proof is analogous to the proof of Lemma~\ref{QVBunfisostrict}. A morphism of pseudo linearized colocales $(\underline f,\ident)$ 
	is an isomorphism if and only if
	$\underline f$ is an isomorphism of colocales and the pair
	$(\underline f^{-1},\ident)$ is a morphism of pseudo linearized colocales. The conditions that 
	both $(\underline f,\ident)$ and $(\underline f^{-1},\ident)$ are morphisms is equivalent to
	$\gamma_2^\complement\circ \underline f_*=\gamma_1^\complement$ in both the covariant and the contravariant cases.\qedhere
\end{proof}

\subsection{The adjunction \boldmath $\clsets\dashv\Iota$}

\subsubsection{From quotient vector bundles to linearized colocales}

Given a continuous map $f\colon X\to Y$  let
\begin{equation}\label{eq:clsetsonmaps}
\clsets f\colon\clsets X\to\clsets Y
\end{equation}
be the map of colocales with inverse image homomorphism $(\clsets f)^{*}=f^{-1}$. Its direct image
homomorphism is then the map $(\clsets f)_*(C)=f^{\complement}_!(C)=\overline{f(C)}$.

\begin{lemma}\label{morQVBuniffmorLinCoLoc}
	Let $(X,A,\kappa_A),(Y,B,\kappa_B)\in\PsiQVBun^{1*}_\Tt$, let $f_\flat\colon X\to Y$ be continuous and let $f^\sharp\colon B\to A$ be linear.
	The following are equivalent:
	\begin{enumerate}
		\item The pair $(f_\flat,f^\sharp )$ satisfies:
		\[
		b\in \kappa_B\bigl(f_\flat(x)\bigr)\ \Rightarrow\ f^\sharp (b)\in\kappa_A(x)\,;
		\]
		\item The pair $(\clsets f_\flat,f^\sharp )$ satisfies:
		\[
		\Max f^\sharp \circ\gamma_B^\complement\circ (f_\flat)_!^\complement\leq \gamma_A^\complement\,.
		\]
	\end{enumerate}
\end{lemma}
\begin{proof}
	From Lemma~\ref{kfrakcircsob=kappa} we have $\kappa_A=\gamma_A^\complement\circ\sob_X^\complement$
	and, using the identity 
	$\sob_Y^\complement\circ f_\flat=(f_\flat)_!^\complement\circ\sob_X^\complement$ we also get:
	\[
	\kappa_B\circ f_\flat
	=\gamma_B^\complement\circ\sob_Y^\complement\circ f_\flat
	=\gamma_B^\complement\circ (f_\flat)_!^\complement\circ\sob_X^\complement
	\]
    It follows that statement 1:
	\[
	\Max f^\sharp \circ\kappa_B\circ f_\flat(x)\subset\kappa_A(x)
	\]
	can be rewritten as
	\begin{equation}\label{eq:sharpgammaflatsubsetgamma}
	\Max f^\sharp \circ\gamma_B^\complement\circ(f_\flat)_!^\complement\bigl(\,\overline{\{x\}}\,\bigr)\subset\gamma_A^\complement\bigl(\,\overline{\{x\}}\,\bigr)\,.
	\end{equation}
	It is now clear that $2\Rightarrow1$. Conversely, if
	$C\in\clsets X$ then:
	\begin{align*}
	\Max f^\sharp \circ\gamma_B^\complement\circ(f_\flat)_!^\complement(C)
	&=\Max f^\sharp \circ\gamma_B^\complement\circ(f_\flat)_!^\complement\Bigl(\bigvee_{x\in C}\overline{\{x\}}\Bigr)\\
	&=\bigvee_{x\in C}\Max f^\sharp \circ\gamma_B^\complement\circ(f_\flat)_!^\complement\bigl(\,\overline{\{x\}}\,\bigr)\\
	&\subset\bigvee_{x\in C}\gamma_A^\complement\bigl(\,\overline{\{x\}}\,\bigr)\quad\text{(using equation~\eqref{eq:sharpgammaflatsubsetgamma})}\\
	&=\gamma_A^\complement\Bigl(\bigvee_{x\in C}\overline{\{x\}}\Bigr)\\
	&=\gamma_A^\complement(C)\,.\qedhere
	\end{align*}
\end{proof}

As a corollary of Lemma~\ref{morQVBuniffmorLinCoLoc} we can now prove:

\begin{corollary}
	The assignements $\mathcal A\to\boldsymbol\clsets\mathcal A$ and
	$(f_\flat,f^\sharp)\to (\clsets f_\flat,f^\sharp)$
	define a functor $\boldsymbol\clsets^{*}$ between the categories $\PsiQVBun^{1*}_\Tt$\ and $\PsiLinCoLoc^{1*}$,
	which restricts to a functor $\QVBun^{1*}\to\LinCoLoc^{1*}$.
\end{corollary}
\begin{proof}
It remains only to show that $\boldsymbol\clsets^{*}$ restricts to a functor $\QVBun^{1*}\to\LinCoLoc^{1*}$.
This immediately follows from Corollary~\ref{kfrakiscontinuous}.\qedhere
\end{proof}

We now look at the covariant case.

\begin{lemma}
	The assignements $\mathcal A\mapsto\clsets\mathcal A$ and $(f_\flat,f^\sharp)\mapsto(\clsets f_\flat,f^\sharp)$
	define a functor $\boldsymbol\clsets_*\colon\PsiQVBun^1_{\Tt*}\to\PsiLinCoLoc_*$ which restricts to a functor
	$\QVBun^{1}_*\to\LinCoLoc_*$.
\end{lemma}
\begin{proof}
	Corollary~\ref{kfrakiscontinuous} shows that $\boldsymbol\clsets_*$ restricts to a functor $\QVBun^{1}_*\to\LinCoLoc^{1}_*$
	so it only remains to show functoriality on morphisms.
	Let $(f_\flat,f^\sharp)$ be a morphism in $\PsiQVBun_{\Tt*}^1$.
	Since $\kappa_B$ has image inside $\Max B$, for any $x\in X$ we have
	\[
	\Max f^\sharp\bigl(\kappa_A(x)\bigr)=\overline{f^\sharp\bigl(\kappa_A(x)\bigr)}
	\subset\overline{\kappa_B\bigl(f_\flat(x)\bigr)}=\kappa_B\bigl(f_\flat(x)\bigr)
	\]
	so, given $C\in\clsets X$:
	\begin{align*}
	\Max f^\sharp\circ\gamma_A^\complement(C)
	&=\Max f^\sharp\Bigl(\bigvee_{x\in C}\kappa_A(x)\Bigr) \\
	&=\bigvee_{x\in C}\Max f^\sharp\bigl(\kappa_A(x)\bigr) \\
	&\subset\bigvee_{x\in C}\kappa_B\bigl(f_\flat(x)\bigr) \\
	&=\bigvee_{y\in f_\flat(C)}\kappa_B(y) \\
	&\subset\bigvee_{y\in \overline{f_\flat(C)}}\kappa_B(y)
	=\gamma_B\bigl(f_{\flat!}^\complement(C)\bigr)
	\end{align*}
	which shows that $(\clsets f_\flat,f^\sharp)$ is a morphism in $\PsiLinCoLoc_*$.\qedhere
\end{proof}

\subsubsection{From linearized colocales to quotient vector bundles}

We first consider the contravariant case.

\begin{lemma}
	The assignements $\mathfrak A^\complement=(L^\complement,A,\gamma^\complement)\mapsto 
	\boldsymbol\Iota\mathfrak A^\complement=(\Iota L^\complement,A,\kfrak^\complement)$
	$\mathfrak A^\complement\to\boldsymbol\Iota\mathfrak A^\complement$ and
	$(\underline f,\overline f)\to (\underline f_*,\overline f)$
	define a functor $\boldsymbol\Iota^{*}\colon\PsiLinCoLoc^{1*}\to\PsiQVBun^{1*}_\Tt$ which restricts to a functor 
	$\LinCoLoc^{1*}\to\QVBun^{1*}$.
\end{lemma}	
\begin{proof}
	Lemma~\ref{lemma:kfrak-1U0,Vopen} shows that $\boldsymbol\Iota\mathfrak A^\complement\in\PsiQVBun^{1*}_\Tt$
	so we only need to check functoriality on morphisms.
	Let $\mathfrak f=(\underline f,\overline f)\colon\mathfrak B^\complement\to\mathfrak A^\complement$ be a morphism in $\PsiLinCoLoc^{1*}$
	and write $\mathfrak A^\complement=( L_A^\complement,A,\sigma_A^\complement)$ and
	$\mathfrak B^\complement=( L_B^\complement,B,\sigma_B^\complement)$.
	For all $a\in A$ and $y\in\Iota L^\complement_B$, we have:
	\begin{align*}
	a\in\kfrak_A^\complement(\underline f_*(y))\
	&\Leftrightarrow\ a\in\gamma_A^\complement\circ \underline f_*(y)\\
	&\Rightarrow\
	\overline f(a)\in\overline f\bigl(\gamma_A^\complement\circ \underline f_*(y)\bigr)\\
	&\Rightarrow\
	\overline f(a)\in \Max\overline f\circ\gamma_A^\complement\circ \underline f_*(y) &
	&\text{(since $\overline f(V)\subset \Max\overline f(V)$)} \\
	&\Rightarrow\ \overline f(a)\in\gamma_B^\complement(y) &
	&\text{(equation~\eqref{eq:morofcolocales})}\\
	&\Leftrightarrow\ \overline f(a)\in\kfrak_B^\complement(y)
	\end{align*}
	so the pair $(\underline f_*,\overline f)\colon \boldsymbol\Iota\mathfrak B^\complement\to\boldsymbol\Iota\mathfrak A^\complement$
	is a morphism in $\PsiQVBun^{1*}_\Tt$.\qedhere
\end{proof}

\begin{lemma}
	The assignements $\mathfrak A\mapsto\Iota\mathfrak A$ and $(\underline f,\overline f)\mapsto(\underline f_*,\overline f)$
	define a functor $\boldsymbol\Iota_*\colon\PsiLinCoLoc_*\to\PsiQVBun_{\Tt*}^1$ which restricts to
	a functor $\LinCoLoc_*\to\QVBun_{*}^1$
\end{lemma}
\begin{proof}
	We only need to check functoriality on morphisms.
	Let $(\underline f,\overline f)$ be a morphism in $\PsiLinCoLoc_*$.
	Let $c\in\Iota L^\complement_A$. Then
	\begin{align*}
	\overline f\bigl(\kfrak_A^\complement(c)\bigr)
	&\subset\overline{\overline f\bigl(\kfrak_A^\complement(c)\bigr)} \\
	&=\Max\overline f\circ\gamma_A^\complement(c) \\
	&\subset\gamma_B^\complement\circ\underline f_*(c) \\
	&=\kfrak_B^\complement\bigl(\underline f_*(c)\bigr)
	\end{align*}
	which shows that $(\underline f_*,\overline f)$ is a morphism in $\PsiQVBun^1_{\Tt*}$.\qedhere
\end{proof}

\subsubsection{Spatialization}

Consider the map $\spat_{L^\complement}^\complement$ in equation~\eqref{eq:spatc}.

\begin{lemma}
	Given a pseudo linearized colocale $\mathfrak A^\complement$, the map $\spat_{L^\complement}^\complement$ is an $\ident$-morphism 
	$\boldsymbol\clsets\boldsymbol\Iota\mathfrak A^\complement\to \mathfrak A^\complement$.
\end{lemma}
\begin{proof}
	Write $\boldsymbol{\clsets}\boldsymbol\Iota\mathfrak A^\complement=(\clsets\Iota L^\complement,A,\widetilde\gamma^\complement)$
	and let $\sigma^\complement$ and $\widetilde\sigma^\complement$ be the right adjoints of
	$\gamma^\complement$ and $\widetilde\gamma^\complement$ respectively.
	We need to show that $\gamma^\complement\circ (\spat_{L^\complement}^\complement)_*=\widetilde\gamma^\complement$.
	Taking adjoints we get the equivalent identity
	$\widetilde\sigma^\complement=(\spat_{L^\complement}^\complement)^{*}\circ\sigma^\complement$ which we now prove.
	By definition $\widetilde\sigma^\complement(V)=(\kfrak^{\complement})^{-1}(\sob_{\Max A}^\complement V)$
	(equation~\eqref{eq:defcosupport}). Since $\sob_{\Max A}^\complement V=\mathcal F_{0,V}$
	(equation~\eqref{eq:soblVC(V)=F_0,V}) we have, for all $V\in\Max A$:
	\begin{equation}\label{eq:proofthatspatisstrict}
	\begin{aligned}
	c\in\widetilde\sigma^\complement(V)\ 
	&\Leftrightarrow\ c\in (\kfrak^{\complement})^{-1}(\mathcal F_{0,V})\\ 
	&\Leftrightarrow\ \gamma^\complement(c)\subset V\\ 
	&\Leftrightarrow\ c\leq\sigma^\complement(V)\\ 
	&\Leftrightarrow\ c\in (\spat_{L^\complement}^\complement)^{*}\bigl(\sigma^\complement(V)\bigr)
	\quad\text{(equation~\eqref{eq:spatc*=Ca})}
	\end{aligned}
	\end{equation}
	so $\widetilde\sigma^\complement=(\spat_{L^\complement}^\complement)^{*}\circ\sigma^\complement$.\qedhere
\end{proof}

\begin{corollary}\label{tildesigma=spatcircsigma}
	Let $\mathfrak A^\complement=( L^\complement,A,\gamma^\complement)$ be a pseudo linearized colocale. Then the pair
	\[
	\spat_{\mathfrak A^\complement}^\complement=(\spat_{L^\complement}^\complement,\ident)
	\colon\boldsymbol\clsets\boldsymbol\Iota\mathfrak A^\complement\to \mathfrak A^\complement
	\]
	is both a covariant and a contravariant morphism of pseudo linearized colocales which is an isomorphism
	if and only if $L^\complement$ is spatial.
\end{corollary}
\begin{proof}
It immediately follows from Lemma~\ref{morlincolocisostrict}.\qedhere
\end{proof}

Given a pseudo linearized colocale $\mathfrak A^\complement=( L^\complement,A,\gamma^\complement)$ we call \emph{spatialization}
to the morphism
$\spat_{\mathfrak A^\complement}^\complement=(\spat_{L^\complement}^\complement,\ident)\colon\boldsymbol\clsets\boldsymbol\Iota\mathfrak A^\complement\to \mathfrak A^\complement$.
We say a pseudo linearized colocale  $\mathfrak A^\complement=( L^\complement,A,\gamma^\complement)$ is \emph{spatial} if $L^\complement$ is spatial.

\subsubsection{Adjunction and equivalence}

We now check that soberification (see Lemma~\ref{lemma:sobAisstrictmor}) is well defined in the covariant case.

\begin{lemma}
	For any pseudo quotient vector bundle $\mathcal A=(X,A,\kappa)$, the map 
	$\sob_{\mathcal A}^\complement=(\sob_X^\complement,\ident)\colon\mathcal A\to\boldsymbol\Iota\boldsymbol\clsets\mathcal A$
	is a morphism in $\PsiQVBun_*$ which is an isomorphism if $X$ is sober and $\mathcal A\in\PsiQVBun_{\Tt*}$.
\end{lemma}
\begin{proof}
	To show that $\sob_{\mathcal A}^\complement$ is a morphism in $\PsiQVBun_*$ 
	we have to show that $\kappa(x)\subset\kfrak^\complement\bigl(\sob_X^\complement(x)\bigr)$. This immediately follows from Lemma~\ref{lemmagamma=bigvee}
	since $x\in\sob_X^\complement(x)$. Now suppose $X$ is sober and $\mathcal A\in\PsiQVBun_{\Tt*}$.
	Then $\sob_X$ is a homeomorphism and by Lemma~\ref{kfrakcircsob=kappa} we have 
	$\kappa(x)=\kfrak^\complement\bigl(\sob_X^\complement(x)\bigr)$ so the pair $\bigl(\sob_X^{-1},\ident\bigr)$
	defines a morphism in $\PsiQVBun_*$ which is the inverse of $\sob_{\mathcal A}^\complement$.\qedhere
\end{proof}

\begin{theorem}\label{T5.10}
	The functors $\boldsymbol\clsets^{*}\colon\PsiQVBun^{1*}_\Tt\to\PsiLinCoLoc^{1*}$ and
	$\boldsymbol\clsets_{*}\colon\PsiQVBun^{1}_{\Tt*}\to\PsiLinCoLoc^{1}_*$
	are left adjoint to, respectively, the functors
	$\boldsymbol\Iota^{*}\colon\PsiLinCoLoc^{1*}\to\PsiQVBun^{1*}_{\Tt}$ and
	$\boldsymbol\Iota_{*}\colon\PsiLinCoLoc^{1}_*\to\PsiQVBun^{1}_{\Tt*}$.
	These adjunctions restrict, in both the covariant and the contravariant cases, to isomorphisms
	between the categories of pseudo quotient vector bundles with Hausdorff fibers over sober base spaces and the categories of linearized colocales
	over a spatial colocale.
\end{theorem}
\begin{proof}
	It is enough to observe that 
	soberification and spatialization are both natural transformations between respectively the identity and 
	$\boldsymbol\Iota\circ\boldsymbol\clsets$,
	and $\boldsymbol\clsets\circ\boldsymbol\Iota$ and the identity,
	satisfying for each cospectral pseudo quotient vector bundle $\mathcal A$ 
	and each pseudo linearized colocale $\mathfrak A^\complement$ the relations
	\[
	\spat_{\boldsymbol\clsets\mathcal A}^\complement\circ\boldsymbol\clsets\sob_{\mathcal A}^\complement=\ident\quad\text{and}\quad
	\boldsymbol\Iota\spat_{\mathfrak A}^\complement\circ\sob_{\boldsymbol\Iota\mathfrak A}^\complement=\ident\,.\qedhere
	\]
\end{proof}

\subsection{The adjunction \boldmath $\Omega\dashv\Sigma$ in the covariant case}

Given a continuous map $f\colon X\to Y$ denote by $\Omega f\colon\clsets X\to\clsets Y$ the map of colocales
with inverse image homomorphism $(\Omega f)^{*}=f^{-1}$.

\begin{lemma}
  The assignements $\mathcal A\mapsto\Omega\mathcal A$ and $(f_\flat,f^\sharp)\mapsto(\Omega f_\flat,f^\sharp)$
  define a functor $\boldsymbol\Omega_*\colon\PsiQVBun_*\to\PsiLLoc_*$ which restricts to a functor $\QVBun_{\Sigma*}\to\LLoc_*$.
\end{lemma}
\begin{proof}
	We only need to check functoriality on morphisms.
Let $(f_\flat,f^\sharp)$ be a morphism in $\PsiQVBun_{*}$.
We want to show that $(\Omega f_\flat,f^\sharp)$ is a morphism in $\LLoc_*$, that is:
\[
\gamma_A\leq(f^\sharp)^{-1}\circ\gamma_B\circ (f_{\flat})_!
\]
Taking adjoints we get the equivalent relation
\[
f_\flat^{-1}\circ \sigma_B\circ\Sub f^\sharp\leq\sigma_A
\]
Let $V\in\Sub A$. Then
\[
\sigma_A(V)=\bigcup_{a\in V}\supp^\circ\widehat a
\]
and
\begin{align*}
  f_\flat^{-1}\circ \sigma_B\circ\Sub f^\sharp(V)
  &=f_\flat^{-1}\Bigl(\bigcup_{b\in f^\sharp(V)}\supp^\circ\widehat b\Bigr) \\
  &=\bigcup_{a\in V}f_\flat^{-1}\bigl(\supp^\circ\widehat{f^\sharp(a)}\bigr)
\end{align*}
so we only need to show that $f_\flat^{-1}\bigl(\supp^\circ\widehat{f^\sharp(a)}\bigr)\subset\supp^\circ\widehat a$
for any $a\in V$. Since $(f_\flat,f^\sharp)$ is a morphism in $\PsiQVBun_{\Sigma*}$ we have,
for any $x\in X$ and any $a\in A$:
\[
a\in\kappa_A(x)\ \Rightarrow\ f^\sharp a\in\kappa_B(f_\flat(x))
\]
so
\begin{align*}
  f_\flat^{-1}\supp^\circ\widehat{f^\sharp(a)}&=f_\flat^{-1}\bigl(\interior\bigl\{y\in Y:f^\sharp(a)\notin\kappa_B(y)\bigr\}\bigr)\\
  &\subset\interior f_\flat^{-1}\bigl\{y\in Y:f^\sharp(a)\notin\kappa_B(y)\bigr\}\\
  &=\interior\bigl\{x\in X:f^\sharp(a)\notin\kappa_B\bigl(f_\flat(x)\bigr)\bigr\}\\
  &\subset\interior\{x\in X:a\notin\kappa_A(x)\}\\
  &=\supp^\circ\widehat a\qedhere
\end{align*}
\end{proof}

\begin{lemma}
  The assignements $\mathfrak A\mapsto\boldsymbol\Sigma\mathfrak A$ and $(\underline f,\overline f)\mapsto(\underline f_*,\overline f)$
  define a functor $\boldsymbol\Sigma_*\colon\PsiLLoc_*\to\PsiQVBun_{\Sigma*}$ which restricts to a functor
  $\LLoc_*\to\QVBun_{\Sigma*}$
\end{lemma}
\begin{proof}
  Let $\mathfrak A=(L,A,\gamma)$ be a pseudo linearized locale. We first show that $\boldsymbol\Sigma\mathfrak A$ has the open support property.
  This is Lemma~5.14 in \cite{ReSa16}, whose proof does not need the continuity of $\kfrak$. Alternatively,
  the statement follows from Lemma~\ref{lemmaOpenSupportiffpseudo} and the observation that 
  $\kfrak^{-1}(\mathcal U_{V,1})=\spat^{*}\sigma(V)$ is open.
  
Now let $(\underline f,\overline f)$ be a morphism in $\PsiLLoc_*$ and let $p\in\Sigma L_A$. Then
\begin{align*}
  \kfrak_A(p)
  &=\gamma_A(p)\\
  &\subset\overline f^{-1}\bigl(\gamma_B\bigl(\underline f_*(p)\bigr)\bigr)\\
  &=\overline f^{-1}\bigl(\kfrak_B\bigl(\underline f_*(p)\bigr)\bigr)
\end{align*}
so $\overline f\bigl(\kfrak_A(p)\bigr)\subset\kfrak_B\bigl(\underline f_*(p)\bigr)$ an it follows that
$(\underline f_*,\overline f)$ is a morphism in $\PsiQVBun_{\Sigma*}$.\qedhere
\end{proof}

\begin{lemma}
  For any pseudo quotient vector bundle
  $\mathcal A=(X,A,\kappa)$, the map $\sob_A=(\sob_X,\ident)\colon\mathcal A\to\boldsymbol\Sigma\boldsymbol\Omega\mathcal A$ is a morphism in $\PsiQVBun_*$.
\end{lemma}
\begin{proof}
  We need to show that $\kappa(x)\subset\kfrak\bigl(\sob_X(x)\bigr)$, for any $x\in X$. This is Lemma~5.6 in \cite{ReSa16}.\qedhere
\end{proof}

\begin{lemma}
	Let $\mathfrak A=(L,A,\gamma)$ be a pseudo linearized locale. Then spatialization  $(\spat_L,\ident)$
	is an $\ident$-morphism $\spat_L\colon\boldsymbol{\Omega\Sigma}\mathfrak A\to\mathfrak A$.
\end{lemma}
\begin{proof}
	This is Theorem 5.23 in~\cite{ReSa16}.\qedhere
\end{proof}

From Lemma~\ref{morlincolocisostrict} it follows that 
the pair $(\spat_L,\ident)$ is both a covariant and a contravariant morphism which is an isomorphism 
if and only if the locale $L$ is spatial.
	
\begin{theorem}\label{T5.15}
	The functor $\boldsymbol\Omega_*\colon\PsiQVBun_{\Sigma*}\to\PsiLLoc_*$ is left adjoint to the functor
	$\boldsymbol\Sigma_*\colon\PsiLLoc_*\to\PsiQVBun_{\Sigma*}$. 
	This adjunction restricts to an isomorphism
	between the category of pseudo quotient vector bundles with Hausdorff fibers over sober base spaces and the category of linearized locales
	over a spatial locale.
\end{theorem}
\begin{proof}
		It is enough to observe that 
		soberification and spatialization are both natural transformations between respectively the identity and 
		$\boldsymbol\Sigma\circ\boldsymbol\Omega$,
		and $\boldsymbol\Omega\circ\boldsymbol\Sigma$ and the identity,
		satisfying for each $\mathcal A\in\PsiQVBun^1_*$ and $\mathfrak A\in\PsiLinCoLoc^1_*$ the relations
		\[
		\spat_{\boldsymbol\Omega\mathcal A}\circ\boldsymbol\Omega\sob_{\mathcal A}=\ident\quad\text{and}\quad
		\boldsymbol\Sigma\spat_{\mathfrak A}\circ\sob_{\boldsymbol\Sigma\mathfrak A}=\ident\,.\qedhere
		\]
\end{proof}

\section{Duality}\label{sec:dualQVBun}\label{sec:6}

\subsection{Dual pairs}

A dual pair is a pair $(A,A^\vee)$ of vector spaces together with a non-degenerate bilinear form
$\langle,\rangle\colon A\times A^\vee\to\mathbb C$. 
If $A$ is a locally convex topological vector space and
$A^\vee$ is the dual space of continuous linear functionals on $A$ then, by the Hahn Banach theorem the induced bilinear form $A\times A^\vee\to\mathbb C$
is non-degenerate so $(A,A^\vee)$ is a dual pair.
A topology on $A$ is said to be a \emph{dual topology} if $A$ is locally convex
and $A^\vee$ is its space of continuous linear functionals.
A topology on $A^\vee$ is a dual topology if $A^\vee$ is locally convex and $A$ is the space
of continuous linear functionals $A^\vee\to\mathbb C$. The weak topology on $A^\vee$ determined by $A$ 
and the weak topology on $A$ determined by $A^\vee$ are the weakest dual topologies
(see for example \cite[chapter 8]{Jar81}).

\subsubsection{Galois correspondence between \boldmath $\Sub A$ and $\Sub A^\vee$}

Given a dual pair $(A,A^\vee)$,
for each $\phi\in A^\vee$ let $\ker\phi=\{a\in A:\langle a,\phi\rangle=0\}$ and
for each $a\in A$ let $\ker a=\{\phi\in A^\vee:a\in\ker\phi\}$.
Consider the annihilator lattice maps $F\colon(\Sub A)^\opp\to\Sub A^\vee$ and $G\colon\Sub A^\vee\to(\Sub A)^\opp$ defined by
\begin{equation}\label{eq:defi(V)r(V)}
\begin{aligned}
F(V)&=\{\phi\in A^\vee:V\subset\ker\phi\}=\bigcap_{a\in V}\ker a\,; \\
G(\mathcal V)&=\{a\in A:\mathcal V\subset\ker a\}=\bigcap_{\phi\in\mathcal V}\ker\phi\,.
\end{aligned}
\end{equation}
Note that, for any $\phi\in A^\vee$, we have:
\begin{equation}\label{eq:F(kerphi)=<phi>}
\begin{aligned}
F(\ker\phi)&=\{\psi\in A^\vee:\ker\phi\subset\ker\psi\}=\langle\phi\rangle\\
G(\langle\phi\rangle)&=\bigcap_{\psi\in\langle\phi\rangle}\ker\psi=\ker\phi
\end{aligned}
\end{equation}
For any $V\in\Sub A$ and any $\mathcal W\in\Sub A^\vee$ we have $V\subset G(\mathcal W)\Leftrightarrow \mathcal W\subset F(V)$
so we have an adjunction $F\dashv G$ which gives a Galois correspondence between $\Sub A$ and $\Sub A^\vee$.
The image of $G$ is the inf-sublattice of $\Sub A$ generated by $\{\ker\phi\}_{\phi\in A^\vee}$. When $A$ has a dual topology, 
the image of $G$ equals $\Max A$. Similarly, if $A^\vee$ has a dual topology then the image of $F$ equals $\Max A^\vee$.
By analogy with this case, we will always write
\[
\image G=\Max A\,,\qquad \image F=\Max A^\vee
\]
and we represent the so called closure operators by
\[
G\circ F(V)=\overline V\quad(V\in\Sub A),\qquad
F\circ G(\mathcal V)=\overline{\mathcal V}\quad(\mathcal V\in\Sub A^\vee)
\]

\begin{lemma}\label{le:iV=ibarV}
	For any $V\in\Sub A$ and any $\mathcal W\in\Sub A^\vee$ we have $F(\overline V)=F(V)$ and $G(\overline{\mathcal W})=G(\mathcal W)$.
\end{lemma}
\begin{proof}
	It immediately follows from the identities $F\circ G\circ F=F$ and $G\circ F\circ G=G$.\qedhere
\end{proof}

\subsubsection{Morphisms of dual pairs}

A morphism of dual pairs $(A,A^\vee)\to(B,B^\vee)$ is a pair of linear maps
$f\colon A\to B$ and $f^\vee\colon B^\vee\to A^\vee$ such that, for every $a\in A$ and every $\psi\in B^\vee$ we have $\langle f(a),\psi\rangle=\langle a,f^\vee(\psi)\rangle$.

\begin{lemma}\label{le:if=f-1i}
	Let $(f,f^\vee)\colon (A,A^\vee)\to (B,B^\vee)$ be a morphism of dual pairs. Then:
	\begin{enumerate}
		\item For any $V\in\Sub A$ and any $W\in\Sub A^\vee$ we have
		\[
		F\bigl(f(V)\bigr)=(f^\vee)^{-1}\bigl(F(V)\bigr)\quad\text{and}\quad G\bigl(f^\vee(\mathcal W)\bigr)=f^{-1}(G(\mathcal W)\bigr)
		\]
		\item For any $W\in\Sub B$ we have
		\[
		F\bigl(f^{-1}(W)\bigr)\supset f^\vee\bigl(F(W)\bigr)
		\]
	\end{enumerate}
\end{lemma}
\begin{proof}
	We first prove (1). For the first identity we have
	\begin{align*}
	F\bigl(f(V)\bigr)&=\bigl\{\phi\in B^\vee:f(V)\subset\ker\phi\bigr\}\\
	&=\bigl\{\phi\in B^\vee:\phi\bigl(f(a)\bigr)=f^\vee(\phi)(a)=0\text{ for all $a\in V$}\bigr\}\\
	&=\bigl\{\phi\in B^\vee:V\subset\ker f^\vee(\phi)\bigr\}\\
	&=\bigl\{\phi\in B^\vee:f^\vee(\phi)\in F(V)\bigr\}\\
	&=(f^\vee)^{-1}\bigl(F(V)\bigr)\,,
	\end{align*}
	and the second identity is proved similarly.
	We now prove (2). From $W\subset f\bigl(f^{-1}(W)\bigr)$ and from (1) we get
	\[
	F(W)\subset F\bigl(f\bigl(f^{-1}(W)\bigr)\bigr)=(f^\vee)^{-1}\bigl(F\bigl(f^{-1}(W)\bigr)\bigr)
	\]
	which is equivalent to identity (2).\qedhere
\end{proof}

Note that, since the pairing is non-degenerate, the map $f^\vee$ is completely determined by $f$.
Also observe that
\begin{equation}\label{f-1kerphi=kerfveephi}
f^{-1}\bigl(\ker\phi)=\{b\in B:\langle f(b),\phi\rangle=\langle b,f^\vee(\phi)\rangle=0\}=\ker f^\vee(\phi)
\end{equation}
The following lemma provides an alternative characterization of morphisms of dual pairs:

\begin{lemma}\label{le:morinpsialternate}
	Let $(A,A^\vee)$ and $(B,B^\vee)$ be dual pairs. Given a linear map $f\colon A\to B$ such that $f^{-1}(V)\in\Max A$ for every $B\in\Max B$,
	there is a unique linear map $f^\vee\colon B^\vee\to A^\vee$ such that $(f,f^\vee)\colon (A,A^\vee)\to(B,B^\vee)$ is a morphism of dual pairs.
	Conversely, given a morphism of dual pairs $(f,f^\vee)$, the linear map $f$ satisfies $f^{-1}(V)\in\Max A$ for every $B\in\Max B$.
\end{lemma}
\begin{proof}
	If $(f,f^\vee)$ is a morphism of dual pairs then, by Lemma~\ref{le:if=f-1i}, $f^{-1}(V)\in\Max A$ for every $B\in\Max B$.
	Conversely, suppose $f^{-1}(V)\in\Max A$ for every $B\in\Max B$.
	Let $A_{\mathrm{alg}}^\vee$ and $B_{\mathrm{alg}}^\vee$ be the algebraic dual spaces of $A$ and $B$ respectively 
	and let $f^\vee\colon B_{\mathrm{alg}}^\vee\to A_{\mathrm{alg}}^\vee$ be the dual of $f$. Non-degeneracy implies we have inclusions
	$A^\vee\to A_{\mathrm{alg}}^\vee$ and $B^\vee\to B_{\mathrm{alg}}^\vee$ so we only need to show that $f^\vee(\phi)\in A^\vee$ for any $\phi\in B^\vee$.
	Let $\phi\in B^\vee$. By equation~\eqref{eq:F(kerphi)=<phi>} we have $\ker\phi\in\Max B$
	so by assumption $f^{-1}(\ker\phi)\in\Max A$ so
	there is some $\mathcal W\in\Sub A^\vee$ such that $f^{-1}(\ker\phi)=G(\mathcal W)$ and from
	equation~\eqref{f-1kerphi=kerfveephi} we get $\ker f^\vee(\phi)=G(\mathcal W)$.
	Consider the annihilator maps
	$F_{\mathrm{alg}}\colon\Sub A\to\Sub A_{\mathrm{alg}}^\vee$ and $G_{\mathrm{alg}}\colon\Sub A_{\mathrm{alg}}^\vee\to\Sub A$.
	From equation~\eqref{eq:F(kerphi)=<phi>} we get 
	\[
	\langle f^\vee(\phi)\rangle=F_{\mathrm{alg}}\bigl(\ker f^\vee(\phi)\bigr)=F_{\mathrm{alg}}\bigl(G_{\mathrm{alg}}(\mathcal W)\bigr)\supset\mathcal W
	\]
	which shows that $f^\vee(\phi)\in A^\vee$.\qedhere
\end{proof}

In particular, a morphism of dual pairs $(f,f^\vee)\colon(A,A^\vee)\to(B,B^\vee)$ induces an inf-lattice homomorphism $f^{-1}\colon \Max B\to \Max A$ 
whose left adjoint is
the sup-lattice homomorphism $\Max f\colon \Max A\to \Max B$ defined by
\begin{equation}\label{eq:def:cf}
\Max f(V)=\overline{f(V)}
\end{equation}

\subsubsection{Continuity of \boldmath $F$ and $G$}

\begin{lemma}\label{i-1(U)andr-1(U)}
	For any $V\in\Sub A$ and any $\mathcal W\in\Sub A^\vee$ we have
	\[
	F^{-1}(\mathcal U_{\mathcal W,1})=\mathcal U_{0,G(\mathcal W)}\quad\text{and}\quad
	G^{-1}(\mathcal U_{V,1})=\mathcal U_{0,F(V)}
	\]
\end{lemma}
\begin{proof}
	Since $\mathcal W\subset F(V)\Leftrightarrow V\subset G(\mathcal W)$ we get
	\[
	V\in F^{-1}(\mathcal U_{\mathcal W,1})\Leftrightarrow \mathcal W\not\subset F(V)
	\Leftrightarrow V\not\subset G(\mathcal W) \Leftrightarrow V\in\mathcal U_{0,G(\mathcal W)}
	\]
	and similarly for the identity $G^{-1}(\mathcal U_{V,1})=\mathcal U_{0,F(V)}$.
\end{proof}

\begin{corollary}
	For any $V\in\Max A$ and any $\mathcal W\in\Max A^\vee$ we have
	\[
	V\in\mathcal U_{G(\mathcal W),1}\Leftrightarrow V\in F^{-1}\bigl(\mathcal U_{0,\mathcal W}\bigr)
	\quad\text{and}\quad
	\mathcal W\in\mathcal U_{F(V),1}\Leftrightarrow \mathcal W\in G^{-1}\bigl(\mathcal U_{0,V}\bigr)
	\]
\end{corollary}
\begin{proof}
	If $\mathcal W\in\Max A^\vee=\image F$ then $\mathcal W=F\bigl(G(\mathcal W)\bigr)$ so, from 
	the second identity in Lemma~\ref{i-1(U)andr-1(U)} (with $V=G(\mathcal W)$) we get
	\[
	F^{-1}(\mathcal U_{0,\mathcal W})=F^{-1}\bigl(\mathcal U_{0,F(G(\mathcal W))}\bigr)=F^{-1}\bigl(G^{-1}(\mathcal U_{G(\mathcal W),1})\bigr)
	\]
	and the first identity follows since $G\circ F=\ident$ on $\Max A$. The proof of the second identity is completely analogous.\qedhere
\end{proof}

\begin{corollary}\label{cor:FGcontinuous}
	Let $(\Sub A)_\Sigma$ and $(\Sub A)_{\mathcal T}$ represent $\Sub A$ with respectively 
	the COS topology and with the dual COS topology.
	The following maps are continuous:
	\begin{align*}
	F&\colon(\Sub A)_{\mathcal T}\to(\Sub A^\vee)_\Sigma &
	G&\colon(\Sub A^\vee)_{\mathcal T}\to(\Sub A)_\Sigma \\
	F&\colon(\Max A)_\Sigma\to(\Max A^\vee)_{\mathcal T} &
	G&\colon(\Max A^\vee)_\Sigma\to(\Max A)_{\mathcal T}
	\end{align*}
	In particular the maps $F$ and $G$ define homeomorphisms $(\Max A)_\Sigma\cong(\Max A^\vee)_{\mathcal T}$ and
	$(\Max A)_{\mathcal T}\cong(\Max A^\vee)_\Sigma$.
\end{corollary}

Let $S_a=\{\phi\in A^\vee:\phi(a)\neq 0\}$ and let $\widetilde S_a=\{\mathcal V\in\Sub A^\vee:\mathcal V\cap S_a\neq\emptyset\}$. Also recall
(equation~\eqref{eq:defnofchecka}) that $\CHECK a=\{V\in\Sub A:a\notin V\}$.

\begin{corollary}\label{inverseimageofcheckequalsU0,ker}
	For any $a\in A$ and any $\phi\in A^\vee$ we have
	\[
	F^{-1}(\CHECK\phi)=\mathcal U_{0,\ker\phi}\quad\text{and}\quad G^{-1}(\CHECK a)=\mathcal U_{0,\ker a}=\widetilde S_a
	\]
\end{corollary}
\begin{proof}
	First observe that
	\[
	\mathcal V\in\widetilde S_a\Leftrightarrow\exists_{\phi\in\mathcal V}:\phi(a)\neq 0\Leftrightarrow\mathcal V\not\subset\ker a
	\]
	so $U_{0,\ker a}=\widetilde S_a$. Now
	$\CHECK a=\mathcal U_{\linspan a,1}$ so, applying Lemma~\ref{i-1(U)andr-1(U)} we get
	\[
	G^{-1}(\CHECK a)=\mathcal U_{0,i(\linspan a)}=\mathcal U_{0,\ker a}
	\]
	and similarly for the first identity.\qedhere
\end{proof}

\subsection{Algebraic quotient vector bundles}

An \emph{algebraic quotient vector bundle} is a triple $\mathcal A=(X,(A,A^\vee),\kappa)$ where $X$ is a topological space,
$(A,A^\vee)$ is a dual pair and $\kappa\colon X\to \Max A$ is any map.
From an algebraic quotient vector bundle $\mathcal A=(X,(A,A^\vee),\kappa)$ we get the pseudo quotient vector bundle 
$\mathcal A^{\mathrm{weak}}=(X,A^{\text{weak}},\kappa)$
where $A^{\text{weak}}$ has the weak topology determined by $A^\vee$. 
A covariant (respectively contravariant) morphism $\mathcal A\to\mathcal B$ of algebraic quotient vector bundles
is defined to be a covariant (respectively contravariant)
morphism $(f_\flat,f^\sharp)\colon\mathcal A^{\mathrm{weak}}\to\mathcal B^{\mathrm{weak}}$
of pseudo quotient vector bundles. We denote by $\AQVBun_*^1$ and
$\AQVBun^{1*}$ the categories of algebraic quotient vector bundles with respectively covariant and contravariant morphisms.

From a pseudo quotient vector bundle
$\mathcal B=(X,A,\kappa)$ we get the algebraic quotient vector bundle $\mathcal B^{\text{alg}}=(X,(A,A^\vee),\kappa)$.
The assignements $\mathcal A\mapsto\mathcal A^{\mathrm{weak}}$ and $\mathcal B\mapsto \mathcal B^{\text{alg}}$
define equivalences of categories $\AQVBun_*^1\cong\PsiQVBun_*^1$ and
$\AQVBun^{1*}\cong\PsiQVBun^{1*}$.

An algebraic quotient vector bundle $\mathcal A$ is said to be spectral, respectively cospectral, if the pseudo quotient vector bundle
$\mathcal A^{\text{weak}}$ is spectral, respectively cospectral. We represent
by $\AQVBun_{\Sigma*}^1$,  $\AQVBun^{1*}_\Sigma$ and $\AQVBun_{\Tt*}^1$, $\AQVBun^{1*}_\Tt$ the full subcategories of respectively spectral
and cospectral algebraic quotient vector bundles.

\subsubsection{The codual algebraic quotient vector bundle}

Given an algebraic quotient vector bundle $\mathcal A=(X,(A,A^\vee),\kappa)$,
the \emph{codual algebraic quotient vector bundle} is the triple
\begin{equation}\label{eq:codualquotvectorbundle}
\mathcal A^\vee=(X,(A^\vee,A),\kappa^\vee)\,,\quad\text{where}\quad\kappa^\vee=F\circ\kappa\,.
\end{equation}

\begin{lemma}\label{dualQVBunhasopensupport}
  Let $\mathcal A$ be a quotient vector bundle with Hausdorff fibers.
  If $\mathcal A$ is spectral then $\mathcal A^\vee$ is cospectral. If $\mathcal A$ is cospectral then $\mathcal A^\vee$ is spectral.
\end{lemma}
\begin{proof}
	Suppose $\mathcal A$ is spectral. Then $\kappa^\vee$ is the composition
	\[
	X\xrightarrow{\kappa}(\Max A)_\Sigma\xrightarrow{F}(\Max A^\vee)_\Tt
	\]
	which is continuous by Corollary~\ref{cor:FGcontinuous} so $\mathcal A^\vee$ is cospectral.

	Suppose now that $\mathcal A$ is cospectral. Then for any $\mathcal W\in\Sub A^\vee$ we have
	\[
	(\kappa^\vee)^{-1}(\mathcal U_{\mathcal W,1})=\kappa^{-1}(F^{-1}(\mathcal U_{\mathcal W,1})=\kappa^{-1}(\mathcal U_{0,G(\mathcal W)})
	\]
        so  $\mathcal A^\vee$ is spectral.\qedhere
\end{proof}

\subsubsection{Algebraic linearized locales and colocales}

An \emph{algebraic linearized locale} is a triple $\mathfrak A=(L,(A,A^\vee),\gamma)$ where $L$ is a locale,
$(A,A^\vee)$ is a dual pair and $\gamma\colon L\to \Max A$ is an inf-lattice homomorphism.
An \emph{algebraic linearized colocale} is a triple $\mathfrak A^\complement=(L^\complement,(A,A^\vee),\gamma^\complement)$ 
where $L^\complement$ is a colocale,
$(A,A^\vee)$ is a dual pair and $\gamma^\complement\colon L^\complement\to \Max A$ is a sup-lattice homomorphism.
As for algebraic quotient vector bundles we write 
$\mathfrak A^{\mathrm{weak}}=(L,A^{\text{weak}},\gamma)$ and 
$\mathfrak A^{\complement,\mathrm{weak}}=(L^\complement,A^{\text{weak}},\gamma^\complement)$.
A covariant (respectively contravariant) morphism $\mathfrak A\to\mathfrak B$ of algebraic linearized locales
is defined to be a covariant (respectively contravariant)
morphism $\mathfrak A^{\mathrm{weak}}\to\mathfrak B^{\mathrm{weak}}$ and similarly for algebraic linearized colocales.
We denote by $\ALLoc_*^1$, $\ALinCoLoc_*^1$ and
$\ALLoc^{1*}$,  $\ALinCoLoc^{1*}$
the categories of algebraic linearized locales and colocales with respectively covariant and contravariant morphisms.

If $\mathfrak B=(L,A,\gamma)$, $\mathfrak B^\complement=(L^\complement,A,\gamma^\complement)$ are
respectively a pseudo linearized locale and a pseudo linearized colocale, we write 
$\mathfrak B^{\text{alg}}=(L,(A,A^\vee),\gamma)$ and $\mathfrak B^{\complement,\text{alg}}=(L^\complement,(A,A^\vee),\gamma^\complement)$.
The assignements $\mathcal A\mapsto\mathcal A^{\mathrm{weak}}$ and $\mathcal B\mapsto \mathcal B^{\text{alg}}$
define equivalences of categories $\ALLoc_*^1\cong\PsiLLoc_*^1$ and
$\ALLoc^{1*}\cong\PsiLLoc^{1*}$
and the assignements $\mathcal A^\complement\mapsto\mathcal A^{\complement,\mathrm{weak}}$ and 
$\mathcal B^\complement\mapsto \mathcal B^{\complement,\text{alg}}$
define equivalences of categories $\ALinCoLoc_*^1\cong\PsiLinCoLoc_*^1$ and
$\ALinCoLoc^{1*}\cong\PsiLinCoLoc^{1*}$

The functors {\boldmath $\Omega$, $\Sigma$, $\clsets$ and $\Iota$} induce functors between the algebraic categories
in the obvious way. For example, if $\mathcal A$ is an algebraic quotient vector bundle we let
$\boldsymbol\Omega\mathcal A=\bigl(\boldsymbol\Omega(\mathcal A^{\text{weak}})\bigr)^{\text{alg}}$.

Given an algebraic linearized colocale $\mathfrak A^\complement=(L^\complement,(A,A^\vee),\gamma^\complement)$ we define
the codual algebraic linearized locale by $(\mathfrak A^\complement)^\vee=((L^\complement)^\opp,(A^\vee,A),F\circ\gamma^\complement)$ and
given an algebraic linearized locale $\mathfrak A=(L,(A,A^\vee),\gamma)$ we define the codual algebraic linearized colocale by
$\mathfrak A^\vee=(L^\opp,(A^\vee,A),F\circ\gamma)$.

\subsubsection{Isomorphisms}

Consider the lattice isomorphism 
$\complement\colon(\Omega X)^\opp\to\clsets X$.

\begin{lemma}\label{thm:codualisisoobjects}\label{L6.9}\hfill
	\begin{enumerate}
	\item For any cospectral algebraic quotient vector bundle $\mathcal A$ the map $\complement\colon\boldsymbol\Omega(\mathcal A^\vee)\xrightarrow{\cong}(\boldsymbol\clsets\mathcal A)^\vee$
          is an $\ident$-isomorphism;
	\item For any spectral algebraic quotient vector bundle $\mathcal A$ the map  $\complement\colon\bigl(\boldsymbol\Omega\mathcal A\bigr)^\vee\xrightarrow{\cong}\boldsymbol\clsets\bigl(\mathcal A^\vee\bigr)$
          is an $\ident$-isomorphism;
	\item For any algebraic linearized colocale $\mathfrak A^\complement$ we have 
	$\boldsymbol\Sigma\bigl((\mathfrak A^\complement)^\vee\bigr)\cong(\boldsymbol\Iota\mathfrak A^\complement)^\vee$;
	\item For any algebraic linearized locale $\mathfrak A$ we have
	$\bigl(\boldsymbol\Sigma\mathfrak A)^\vee\cong\boldsymbol\Iota\bigl(\mathfrak A^\vee\bigr)$.
	\end{enumerate}
\end{lemma}
\begin{proof}
	Given an algebraic quotient vector bundle $\mathcal A=(X,(A,A^\vee),\kappa)$ we will represent its support and cosupport maps by
	$\sigma_\kappa$ and $\sigma_\kappa^\complement$ and we will represent its restriction and corestriction maps by
	$\gamma_\kappa$ and $\gamma_\kappa^\complement$.
	
	Let $\mathcal A=(X,(A,A^\vee),\kappa)$ be a cospectral algebraic quotient vector bundle. 
	Then $\boldsymbol\Omega(\mathcal A^\vee)=(\Omega X,(A^\vee,A),\gamma_{F\circ\kappa})$
	and $(\boldsymbol\clsets\mathcal A)^\vee=\bigl((\clsets X)^\opp,(A^\vee,A),F\circ\gamma_\kappa^\complement\bigr)$
	so the first statement will follow if we prove that
	$\gamma_{F\circ\kappa}=F\circ\gamma^\complement_\kappa\circ\complement$. 
	Since, by Lemma~\ref{dualQVBunhasopensupport}, 
	$(X,(A^\vee,A),\kappa^\vee)$ has the open support property, by Lemma~\ref{lemmaOpenSupportiff} we have
	\begin{align*}
	\sigma_{F\circ\kappa}(\mathcal V)
	&=\kappa^{-1}(F^{-1}\bigl(\mathcal U_{\mathcal V,1})\bigr)\\
	&=\kappa^{-1}(\mathcal U_{0,G(\mathcal V)})\quad\text{(Lemma~\ref{i-1(U)andr-1(U)})}\\
	&=\kappa^{-1}\circ\complement(\mathcal F_{0,G(\mathcal V)})\\
	&=\complement\circ\kappa^{-1}(\mathcal F_{0,G(\mathcal V)})
	\end{align*}
	so $\sigma_{F\circ\kappa}=\complement\circ\sigma^\complement_\kappa\circ G$.
	Taking left adjoint we get $\gamma_{F\circ\kappa}=F\circ\gamma^\complement_\kappa\circ\complement$ finishing the proof.
		
	Now let $\mathcal A=(X,(A,A^\vee),\kappa)$ be a spectral algebraic quotient vector bundle. Then
	$\bigl(\boldsymbol\Omega\mathcal A\bigr)^\vee=\bigl((\Omega X)^\opp,(A^\vee,A),F\circ\gamma_\kappa\bigr)$ and 
	$\boldsymbol\clsets\bigl(\mathcal A^\vee\bigr)=\bigl(\clsets X,(A^\vee,A),\gamma^\complement_{F\circ\kappa}\bigr)$ so
	we only need to show that $\gamma^\complement_{F\circ\kappa}=F\circ\gamma_\kappa\circ\complement$.
	\begin{align*}
	\sigma^\complement_{F\circ\kappa}
	&=\kappa^{-1}\bigl(F^{-1}(\mathcal F_{0,V})\bigr)\\
	&=\kappa^{-1}\bigl(\mathcal F_{G(V),1}\bigr)\\
	&=\complement\circ\sigma_\kappa\circ G(V)
	\end{align*}
	The result now follows by taking adjoints.
		
	To prove the third statement let $\mathfrak A^\complement=(L^\complement,(A,A^\vee),\gamma^\complement)$ 
	be an algebraic linearized colocale and
	let $\kfrak^\complement$ denote the restriction of $\gamma^\complement$ to 
	$\Sigma\bigl((L^\complement)^\opp\bigr)=\Iota L^\complement$. Then
	$\boldsymbol\Sigma\bigl((\mathfrak A^\complement)^\vee\bigr)=\bigl(\Sigma (L^\complement)^\opp,(A^\vee,A),F\circ\kfrak^\complement\bigr)$ and
	$(\boldsymbol\Iota\mathfrak A^\complement)^\vee=\bigl(\Iota L^\complement,(A^\vee,A),F\circ\kfrak^\complement\bigr)$ so
	$\boldsymbol\Sigma\bigl((\mathfrak A^\complement)^\vee\bigr)=(\boldsymbol\Iota\mathfrak A^\complement)^\vee$.
	
	Finally let $\mathfrak A=(L,(A,A^\vee),\gamma)$ be an algebraic linearized locale
	and let $\kfrak$ denote the restriction of $\gamma$ to $\Iota L^\opp=\Sigma L$. Then
	$\bigl(\boldsymbol\Sigma\mathfrak A)^\vee=\bigl(\Sigma L,(A^\vee,A),F\circ\kfrak\bigr)$ and
	$\boldsymbol\Iota\bigl(\mathfrak A^\vee\bigr)=\bigl(\Iota L^\opp,(A^\vee,A),F\circ\kfrak\bigr)$ so
	$\bigl(\boldsymbol\Sigma\mathfrak A)^\vee=\boldsymbol\Iota\bigl(\mathfrak A^\vee\bigr)$.\qedhere
\end{proof}

If $(f_\flat,f^\sharp)$ is a morphism of algebraic quotient vector bundles,
by Lemma~\ref{le:morinpsialternate}, the linear map $f^\sharp$ defines a morphism of dual pairs $(f^\sharp,f^{\sharp\vee})$.

\begin{lemma}\label{L6.10}
The assignements $\mathcal A\mapsto\mathcal A^\vee$
and $(f_\flat,f^\sharp)\mapsto(f_\flat,f^{\sharp\vee})$ define
isomorphisms of categories $\AQVBun^1_{\Sigma*}\cong\AQVBun_\Tt^{1*}$ and $\AQVBun^1_{\Tt*}\cong\AQVBun_\Sigma^{1*}$.
\end{lemma}
\begin{proof}
	From Lemma~\ref{thm:codualisisoobjects} it follows that we only need to check the morphisms.
	
Let $(f_\flat,f^\sharp)$ be a morphism in $\AQVBun^{1*}$. Then
\[
f^\sharp\bigl(\kappa_A\bigl(f_\flat(y)\bigr)\bigr)\subset\kappa_B(y)
\]
so
\[
F\bigl(\kappa_B(y)\bigr)\subset
F\bigl(f^\sharp\bigl(\kappa_A\bigl(f_\flat(y)\bigr)\bigr)\bigr)\,.
\]
Applying Lemma~\ref{le:if=f-1i} and noting that $\kappa_B^\vee=F\circ\kappa_B$ we get
\[
\kappa_B^\vee(y)\subset
(f^{\sharp\vee})^{-1}\bigl(\kappa_A^\vee\bigl(f_\flat(y)\bigr)\bigr)
\]
which is equivalent to
\[
f^{\sharp\vee}\bigl(\kappa_B^\vee(y)\bigr)\subset\kappa_A^\vee\bigl(f_\flat(y)\bigr)
\]
so $(f_\flat,f^{\sharp\vee})$ is a morphism in $\AQVBun_*^1$.

Now let
$(f_\flat,f^\sharp)$ be a morphism in $\AQVBun_*^1$. Then
\[
f^\sharp\bigl(\kappa_A(x)\bigr)\subset\kappa_B\bigl(f_\flat(x)\bigr)
\]
so
\[
\kappa_B^\vee\bigl(f_\flat(x)\bigr)=F\bigl(\kappa_B\bigl(f_\flat(x)\bigr)\bigr)\subset
F\bigl(f^\sharp\bigl(\kappa_A(x)\bigr)\bigr)
=(f^{\sharp\vee})^{-1}\bigl(\kappa_A^\vee(x)\bigr)
\]
from whence it follows that
\[
f^{\sharp\vee}\bigl(\kappa_B^\vee\bigl(f_\flat(x)\bigr)\bigr)\subset\kappa_A^\vee(x)
\]
so $(f_\flat,f^{\sharp\vee})$ is a morphism in $\AQVBun^{1*}$.\qedhere
\end{proof}

\begin{lemma}
  The assignements $\mathfrak A\mapsto\mathfrak A^\vee$
  and $(\underline f,\overline f)\mapsto(\underline f,{\overline f}^\vee$ define isomorphisms of categories
  $\ALLoc^{1*}\cong\ALinCoLoc_*^1$ and $\ALLoc^1_*\cong\ALinCoLoc^{1*}$.
\end{lemma}
\begin{proof}
	We only need to check the morphisms.
  Let $(f_\flat,f^\sharp)$ be a morphism in $\ALinCoLoc^{1*}$.
From Lemmas~\ref{le:if=f-1i} and~\ref{le:iV=ibarV} we get
\[
F\circ\Max\overline f=(\overline f{}^{\vee})^{-1}\circ F
\]
In this way
\begin{align*}
  \Max\overline f\circ\gamma_B^\complement\circ\underline f_*\leq\gamma_A^\complement\
  &\Rightarrow\ F\circ\gamma_A^\complement\leq F\circ\Max\overline f\circ\gamma_B^\complement\circ\underline f_* \\
  &\Leftrightarrow\ \gamma_A^{\complement\vee}\leq(\overline f{}^{\vee})^{-1}\circ F\circ\gamma_B^\complement\circ\underline f_* 
\end{align*}
Now let $(\underline f,\overline f)$ be a morphism in $\LLoc_*$. Then
$\gamma_A\leq\overline f{}^{-1}\circ\gamma_B\circ\underline f_*$ so
\[
F\circ\overline f{}^{-1}\circ\gamma_B\circ\underline f_*\leq F\circ\gamma_A
\]
From Lemma~\ref{le:if=f-1i}(2) we get
\begin{align*}
  \Sub\overline f{}^\vee\circ \gamma_B^\vee\circ\underline f_*
  &=\Sub\overline f{}^\vee\circ F\circ \gamma_B\circ\underline f_* \\
  &\leq F\circ\overline f{}^{-1}\circ\gamma_B\circ\underline f_* \\
  &\leq F\circ\gamma_A=\gamma_A^\vee
\end{align*}
\end{proof}

\begin{theorem}\label{T6.12}
	There are natural isomorphisms 
	\begin{align*}
	&\boldsymbol\Omega_*\circ\vee\cong\vee\circ\boldsymbol\clsets^*\colon\AQVBun^{1*}_\Tt\to\ALLoc^{1}_*\\
	&\boldsymbol\Omega^*\circ\vee\cong\vee\circ\boldsymbol\clsets_*\colon\AQVBun^{1}_{\Tt*}\to\ALLoc^{1*}_*\\
	&\boldsymbol\clsets_*\circ\vee\cong\vee\circ\boldsymbol\Omega^*\colon\AQVBun^{1*}_\Sigma\to\ALinCoLoc^{1}_*\\
	&\boldsymbol\clsets^*\circ\vee\cong\vee\circ\boldsymbol\Omega_*\colon\AQVBun^{1}_{\Sigma*}\to\ALinCoLoc^{1*}_*\\
	&\boldsymbol\Sigma_*\circ\vee\cong\vee\circ\boldsymbol\Iota^*\colon\LinCoLoc^{1*}\to\AQVBun^{1}_{\Sigma*}\\
	&\boldsymbol\Sigma^*\circ\vee\cong\vee\circ\boldsymbol\Iota_*\colon\LinCoLoc^{1}_*\to\AQVBun^{1*}_{\Sigma}\\
	&\boldsymbol\Iota_*\circ\vee\cong\vee\circ\boldsymbol\Sigma^*\colon\LLoc^{1*}\to\AQVBun^{1}_{\Tt*}\\
	&\boldsymbol\Iota^*\circ\vee\cong\vee\circ\boldsymbol\Sigma_*\colon\LLoc^{1}_*\to\AQVBun^{1*}_{\Tt}
	\end{align*}
\end{theorem}
\begin{proof}
	It follows easily from Lemma~\ref{thm:codualisisoobjects}.\qedhere
\end{proof}

\section{The codual quotient vector bundle}\label{sec:7}

Given a quotient vector bundle $\mathcal A=(X,A,\kappa)$ 
consider the pseudo quotient vector bundle $\mathcal A^\vee=(X,A^\vee,F\circ\kappa)$.
We wish to investigate when is $\mathcal A^\vee$ a quotient vector bundle.

\subsection{The Fell topology}

We now assume $A^\vee$ has the topology of uniform convergence on compacts, with subbasis the open sets
$S(K,U)=\{\phi\in A^\vee:\phi(K)\subset U\}$ with $K\subset A$ compact and $U\subset\mathbb C$ open.
The objective of this section is to prove the following theorem:

\begin{theorem}\label{thm:Fell<=>Fell}
  Let $\mathcal A=(X,A,\kappa)$ be a quotient vector bundle with Hausdorff fibers and
  suppose $A^\vee$ has the topology of uniform convergence on compacts. Then:
	\begin{enumerate}
		\item $\mathcal A^\vee$ is a quotient vector bundle if
		and only if the kernel map $\kappa$ is Fell continuous;
		\item If $A$ is first countable and the kernel map
		$\kappa$ is Fell continuous then
		$\kappa^\vee$ is also Fell continuous;
		\item If $A$ is normed then
		$\kappa$ is Fell continuous if and only if
		$\kappa^\vee$ is Fell continuous.
	\end{enumerate}
\end{theorem}

We say a quotient vector bundle is Hausdorff if its associated linear bundle is Hausdorff.

\begin{corollary}\label{coro:Fell<=>Fell}
	Let $\mathcal A=(X,A,\kappa)$ be a quotient vector bundle with Hausdorff fibers and suppose both $A$ and $X$ are Hausdorff and first countable.
	If $\mathcal A$ is Hausdorff then $\mathcal A^\vee$ is a Hausdorff quotient vector bundle, with equivalence if $A$
	is normed.
\end{corollary}
\begin{proof}
	It follows immediately from Theorem~\ref{thm:Fell<=>Fell} and \cite[Corollary 5.9]{ReSa17}.\qedhere
\end{proof}

For the proof of Theorem~\ref{thm:Fell<=>Fell} we will
consider the following four topologies on $\Max A$:

\begin{enumerate}
	\item The lower Vietoris topology induced by the topology on $A$.
	\item The topology generated by the sets $\CHECK K$, for $K\subset A$ compact.
	\item The topology generated by the sets $F^{-1}\bigl(\CHECK K{}'\bigr)$, with $K'\subset A^\vee$ compact.
	\item The topology consisting of the sets $F^{-1}(U)$ where $U\subset\Max A^\vee$ is open in the lower Vietoris topology
	induced by the topology on $A^\vee$.
\end{enumerate}

Notice that the Fell topology on $\Max A$ is generated by topologies 1 and 2. The topology generated by 3 and 4 consists of the sets
$F^{-1}(\mathcal U)$ where $\mathcal U\subset \Max A^\vee$ is open in the Fell topology.

\subsubsection{Topologies 2 and 4}

\begin{lemma}\label{lema:VincheckK}\hfill
	\begin{enumerate}
		\item If $K\subset A$ is compact and $V\in\CHECK K$ then
		there are $\phi_1,\dots,\phi_n\in F(V)$ such that $\bigcap_i\ker\phi_i\in\CHECK K$.
		\item If $K$ is compact and convex and  $V\in\CHECK K$ then there is $\phi\in F(V)$ such that $\ker\phi\in\CHECK K$.
		\item Topology 2 has as a subbasis the collection $\{\CHECK K\}$ where $K\subset A$ is a convex compact set.
	\end{enumerate}
\end{lemma}
\begin{proof}
	Let $K\subset A$ be compact and let $V\in\CHECK K$, that is, $V\cap K=\emptyset$. By the Hahn-Banach Theorem
	for each $a\in K$ there is $\phi_a\in F(V)$ such that $\phi_a(a)=1$.
	Let $B=\{z\in\CC:|z-1|<\frac12\}$ and let $U_a=\phi_a^{-1}(B)$.
	Then the collection $\{U_a\}_{a\in K}$ is an open cover of $K$. Pick a finite subcover
	$U_{a_1},\dots,U_{a_n}$. Since $\ker\phi_{a_i}\cap U_{a_i}=\emptyset$ for all $i$, we get $K\cap\bigl(\bigcap_i\ker\phi_{a_i}\bigr)=\emptyset$.
	This proves 1. To prove 3 let $K$, $V$ and $\phi_a$ be as above and
	let $K_i$ be the convex hull of $\phi_{a_i}^{-1}(\overline B)\cap K$. Then
	$K=\bigcup_i\phi_{a_i}^{-1}(\overline B)\cap K\subset\bigcup_iK_i$ so
	$\bigcap_i\CHECK K_i\subset\CHECK K$.
	Also, since $\overline B$ is convex, $\phi_{a_i}(K_i)\subset \overline B$ so $\ker\phi_{a_i}\cap K_i=\emptyset$, and since $V\subset\ker\phi_{a_i}$,
	we get $V\cap K_i=\emptyset$ for all $i$ so $V\in\bigcap_i\CHECK K_i$.
	We showed that $V\in\bigcap_i\CHECK K_i\subset\CHECK K$. This finishes the proof of 3.
	We now prove 2. First consider the special case where $A$ is Hausdorff and has finite 
	dimension, and $V=\{0\}$. Endow $A$ with an inner product and
	let $a\in K$ be the closest point
	to the origin. Then we can take $\phi$ to be inner product with $a$.
	Now consider the general case. By 1 there are $\phi_1,\dots,\phi_n\in F(V)$ with $\bigcap_i\ker\phi_i\in\CHECK K$.
	Let $W=\bigcap_i\ker\phi_i$
	and let $p\colon A\to A/W$ be the projection. Then $p(K)$ is convex and compact, $p(V)=\{0\}$ and $A/W$ is Hausdorff and finite dimensional, so there is
	$\phi\in(A/W)^\vee$ such that $\ker\phi\cap p(K)=\emptyset$. Then $V\subset\ker(\phi\circ p)$ and $\ker(\phi\circ p)\cap K=\emptyset$,
	which finishes the proof of 3.\qedhere
\end{proof}

\begin{corollary}\label{checkK=i-1S}
	For any $K\subset A$ we have 
	$F^{-1}\bigl(\widetilde S(K,\CC\setminus\{0\})\bigr)\subset\CHECK K$
	with equality if $K$ is convex and compact.
\end{corollary}
\begin{proof}
	Suppose $V\in F^{-1}\bigl(\widetilde S(K,\CC\setminus\{0\})\bigr)$. Then there is $\phi\in F(V)$
	such that $\phi(K)\subset\CC\setminus\{0\}$ and hence $K\cap\ker\phi=\emptyset$. 
	Since $V\subset\ker\phi$ it follows that $V\in\CHECK K$. Conversely, assume $K$ is convex and let
	$V\in\CHECK K$. Then by Lemma~\ref{lema:VincheckK}
	there is $\phi\in F(V)$ such that $\ker\phi\cap K=\emptyset$.
	It follows that $\phi\in S(K,\CC\setminus\{0\})$ so $V\in F^{-1}\bigl(\widetilde S(K,\CC\setminus\{0\})\bigr)$.\qedhere
\end{proof}

Let $\mathcal K(A)$ denote the collection of the compact and convex subsets
$K\subset A$ with the property that, if $\lambda\in\CC$ satisfies
$|\lambda|\leq 1$, then $\lambda K\subset K$. 

\begin{lemma}\label{basiswithKinK(A)}
	The topology of compact convergence in $A^\vee$ has as a basis the collection
	\[
	B_{\varepsilon,K}(\phi)=\{\psi\in A^\vee:\max_{a\in K}|\psi(a)-\phi(a)|<\varepsilon\}
	\]
	where $K\in\mathcal K(A)$. 
\end{lemma}
\begin{proof}
	Given a compact set $K$ let $K'$ be the convex hull of the set $\{\lambda a:a\in K,\lambda\in\CC,|\lambda|\leq1\}$. 
	Then $K'\in\mathcal K(A)$ and
	$B_{\varepsilon,K}(\phi)=B_{\varepsilon,K'}(\phi)$.\qedhere
\end{proof}

\begin{lemma}\label{fell2=4}
	Topologies 2 and 4 coincide.
\end{lemma}
\begin{proof}
	Let $K$ be compact and convex. Then by Corollary~\ref{checkK=i-1S} we have
	$\CHECK K=F^{-1}\bigl(\widetilde S(K,\CC\setminus0)\bigr)$ so $\CHECK K$ is open in topology 4.
	By Lemma~\ref{lema:VincheckK} this implies
	topology 4 is finer than topology 2. Now let $U\subset A^\vee$ be open and suppose $V\in F^{-1}\bigl(\widetilde U\bigr)$. Then $F(V)\cap U\neq\emptyset$ so
	by Lemma~\ref{basiswithKinK(A)} there is
	$\phi\in F(V)$, $\varepsilon>0$ and $K\in\mathcal K(A)$ such that $B_{\varepsilon,K}(\phi)\subset U$. We may assume that $\phi\neq 0$ otherwise
	we would have $\widetilde U=\Max A^\vee$.
	To complete the proof we will construct a compact set $K_\varepsilon$ such that $V\in\CHECK K_\varepsilon$
	and $F(\CHECK K_\varepsilon)\subset\widetilde B_{\varepsilon,K}(\phi)\subset\widetilde U$. 
	
	Pick $a\in A$ such that $\phi(a)=1$ and write $A=\ker\phi\oplus\linspan a$. 
	Let $p\colon A\to\ker\phi$ be the projection: $p(b)=b-\phi(b)a$. Let $K_\varepsilon=p(K)-\varepsilon a$. Then 
	for any $b\in K_\varepsilon$ we have $\phi(b)=-\varepsilon$ so $\ker\phi\in \CHECK K_\varepsilon$, and since
	$V\subset\ker\phi$ we also get $V\in\CHECK K_\varepsilon$. It remains to be shown that 
	$F(\CHECK K_\varepsilon)\subset\widetilde B_{\varepsilon,K}(\phi)$. Let $W\in \CHECK K_\varepsilon$.
	Since $K\in\mathcal K(A)$ is convex,  $K_\varepsilon$ is also convex, so by Lemma~\ref{lema:VincheckK} there is $\psi\in F(W)$ such that $\ker\psi\cap K_\varepsilon=\emptyset$.
	Now $0\in K$ so $-\varepsilon a\in K_\varepsilon$ so $\psi(a)\neq 0$ and hence we may rescale $\psi$ so that $\psi(a)=1$.
	Assume by contradiction that $\psi\notin B_{\varepsilon,K}(\phi)$. Then there is $b\in K$ such that
	$|\psi(b)-\phi(b)|=|\psi(p(b))|\geq\varepsilon$ so $\psi(p(b))\neq 0$. Let
	$\lambda=\varepsilon/\psi(p(b))$. Then $|\lambda|\leq 1$ so $\lambda K\subset K$ and hence $\lambda p(b)\in p(K)$ so
	we easily check that $\lambda p(b)-\varepsilon a\in K_\varepsilon\cap\ker\psi$, a contradiction. 
	It follows that $\psi\in B_{\varepsilon,K}(\phi)$ so $F(W)\cap B_{\varepsilon,K}(\phi)\neq\emptyset$, that is,
	$F(W)\in\widetilde B_{\varepsilon,K}(\phi)$. This shows that $F(\CHECK K_\varepsilon)\subset\widetilde B_{\varepsilon,K}(\phi)$, which finishes the proof.\qedhere
	\end{proof} 
	
	We can now prove the first statement in Theorem~\ref{thm:Fell<=>Fell}.
	
	\begin{corollary}
		Let $(X,A,\kappa)$ be a quotient vector bundle with Hausdorff fibers.
		Suppose $A^\vee$ has the compact open topology.
		Then $\mathcal A^\vee$ is a quotient vector bundle if and only if the map $\kappa$
		is continuous relative to the Fell topology on $\Max A$.
	\end{corollary}
	\begin{proof}
		The triple $(X,A^\vee,F\circ\kappa)$ is a quotient vector bundle if and only if the map
		$F\circ\kappa$ is continuous when $\Max A^\vee$ is given the lower Vietoris topology.
		By Lemma~\ref{fell2=4} this is equivalent to $\kappa^{-1}(U)$ being open for any $U$ in topology 2.
		The result follows Since $\kappa$ is continuous with respect to topology 1 and the Fell topology
		on $\Max A$ is the topology generated by topologies 1 and 2.\qedhere
	\end{proof}

	\subsubsection{Topologies 1 and 3}

	\begin{lemma}\label{iwidetildeUsubsetcheckK}
		Let $K\subset A^\vee$ be compact and let $U=A\setminus\bigl(\bigcup_{\phi\in K}\ker\phi\bigr)$.
		\begin{enumerate}
			\item Let $a\in A$. Then $F\bigl(\linspan a\bigr)\in\CHECK K$ if and only if $a\in U$.
			\item Let $V\in\Max A$. Then $V\in\widetilde U$ if and only if there is $a\in V$ such that $F\bigl(\linspan a\bigr)\in\CHECK K$.
			\item We have $F\bigl(\widetilde U\bigr)\subset\CHECK K$.
		\end{enumerate}
	\end{lemma}
	\begin{proof}
		We first prove 1.
		Note that $F\bigl(\linspan a\bigr)=\ker a=\{\psi\in A^\vee:\psi(a)=0\}$. So $F\bigl(\linspan a\bigr)\in\CHECK K$
		if and only if $\ker a\cap K=\emptyset$
		if and only if for all $\psi\in K$ we have $\psi(a)\neq 0$. This is equivalent to $a\notin\bigcup_{\phi\in K}\ker\phi$.
		Statement 2 follows immediately from 1. To prove 3 let $V\in\widetilde U$ and choose $a\in V\cap U$. By 1 we have $F\bigl(\linspan a\bigr)\in\CHECK K$.
		Since $F(V)\subset F\bigl(\linspan a\bigr)$ it follows that $F(V)\in\CHECK K$.\qedhere
	\end{proof}
	
	\begin{lemma}\label{bigcupkerphiisclosed}
		If $A$ is first countable and $K\subset A^\vee$ is compact then $\bigcup_{\phi\in K}\ker\phi$ is closed in $A$.
	\end{lemma}
	\begin{proof}
		Pick a sequence $(a_n)$ in $\bigcup_{\phi\in K}\ker\phi$ and suppose $a=\lim a_n$. We want to show that $a\in\bigcup_{\phi\in K}\ker\phi$.
		For each $n$ there is a $\phi_n\in K$ such that $\phi_n(a_n)=0$. Let $\phi$ be a limit point of the sequence $(\phi_n)$
		and consider the compact set $K'=\{a_n\}_{n\in\mathbb N}\cup\{a\}$. Then %
		for any $\varepsilon>0$ and any $p\in\mathbb N$,
		there is a $n>p$ such that $\phi_n\in B_{\varepsilon,K'}(\phi)$. In particular this means that $|\phi(a_n)|=|\phi(a_n)-\phi_n(a_n)|<\varepsilon$.
		Since $\phi$ is continuous, $\lim\phi(a_n)=\phi(a)$ so we must have $\phi(a)=0$, which concludes the proof.\qedhere
	\end{proof}
	
	\begin{lemma}\label{Fell1isfinerthan3}
		If $A$ is first countable then topology 1 is finer than topology 3.
	\end{lemma}
	\begin{proof}
		Let $K\subset A^\vee$ be compact and let $V\in F^{-1}\bigl(\CHECK K\bigr)$.
		Then $F(V)\cap K=\emptyset$ so $\phi\in K\Rightarrow\phi\notin F(V)\Leftrightarrow V\not\subset\ker\phi$
		and hence for each $\phi\in K$ there is $a_\phi\in V\setminus\ker\phi$ and we may choose $a_\phi$ so that $\phi(a_\phi)=1$.
		Let $B=\{\lambda\in\CC:|\lambda-1|<\frac12\}$. Regarding $a_\phi$ as elements of $A^{\vee\vee}$, we can cover $K$ with the open sets $a_\phi^{-1}(B)$. 
		Since $K$ is compact, there are $a_1,\dots,a_n\in V$ such that
		$K\subset\bigcup_i a_F^{-1}(B)$. Let $K_i=a_F^{-1}(\overline B)\cap K$. Then $K=\bigcup_iK_i$ and $F\bigl(\linspan{a_i}\bigr)\cap K_i=\emptyset$    %
		so $F\bigl(\linspan{a_i}\bigr)\in\CHECK K_i$. 
		Let $U_i=A\setminus\bigl(\bigcup_{\phi\in K_i}\ker\phi\bigr)$ which is open by Lemma~\ref{bigcupkerphiisclosed}.
		Then by Lemma~\ref{iwidetildeUsubsetcheckK} we get $a_i\in U_i$ so $V\in\bigcap_i\widetilde U_i$. To finish the proof we have to show that
		$\bigcap_i\widetilde U_i\subset F^{-1}(\CHECK K)$. This immediately follows from Lemma~\ref{iwidetildeUsubsetcheckK} since
		\[
		F\Bigl(\bigcap_i\widetilde U_i\Bigr)\subset\bigcap_i F\bigl(\widetilde U_i\bigr)\subset\bigcap_i\CHECK K_i=\CHECK K\,.\qedhere
		\]
	\end{proof}

	\begin{lemma}
		Let $U\subset A$ be open and let $C=\{\phi\in A^\vee:\ker\phi\cap U=\emptyset\}$. 
		Then $\linspan a\in\widetilde U$ implies $F\bigl(\linspan a\bigr)\cap C=\emptyset$ with equivalence if $U$ is convex.
	\end{lemma}
	\begin{proof}
		If $F\bigl(\linspan a\bigr)\cap C\neq\emptyset$ there is
		$\phi\in C$ such that $\linspan a\subset\ker\phi$. Since $\ker\phi\cap U=\emptyset$ we get $\linspan a\notin \widetilde U$.
		Conversely, if $U$ is convex and $\linspan a\notin \widetilde U$, then $\linspan a\cap U=\emptyset$ so
		by the Hahn-Banach separation theorem there is $\phi\in A^\vee$
		such that $a\in \ker\phi$ and $\phi\in C$ so $F\bigl(\linspan a\bigr)\cap C=\emptyset$.\qedhere
	\end{proof}
	
	\begin{lemma}\label{Kvarepsilonaiscompact}
		Suppose $A$ is normed and let $S(A^\vee)$ be the unit sphere in $A^\vee$. 
		For any $\varepsilon>0$ and $a\in A$
		the set $K_{\varepsilon,a}=\{\phi\in S(A^\vee):\ker\phi\cap B_\varepsilon(a)=\emptyset\}$
		is compact in both the product topology and the compact open topology.
	\end{lemma}
	\begin{proof}
		Since $S(A^\vee)$ is equicontinuous the pointwise convergence topology and the compact open topology coincide on $S(A^\vee)$.
		Since $S(A^\vee)$ is also compact,
		we only need to check that $K_{\varepsilon,a}$ is closed. But
		$\ker\phi\cap B_\varepsilon(a)=\emptyset$ if and only if $d(a,\ker\phi)\geq\varepsilon$ and
		$d(a,\ker\phi)=|\phi(a)|/\|\phi\|=|\phi(a)|$ \cite[Lemma~1.1]{HNO86}
		and the condition $|\phi(a)|\geq\varepsilon$ defines a closed subset of $A^\vee$.\qedhere
	\end{proof}
	
	\begin{lemma}
		Suppose $A$ is a normed vector space. Then topologies 1 and 3 coincide.
	\end{lemma}
	\begin{proof}
		Let $U\subset A$ be an open set. By Lemma~\ref{Fell1isfinerthan3} we only need to show that $\widetilde U$ is open in topology 3.
		Suppose $V\in\widetilde U$. Then there is $a\in V$ and $\varepsilon>0$ such that $B_\varepsilon(a)\subset U$. Let
		\[
		K=\bigl\{\phi\in S(A^\vee):\ker\phi\cap B_\varepsilon(a)=\emptyset\bigr\}
		\]
		which is compact by Lemma~\ref{Kvarepsilonaiscompact}. First we show that $V\in F^{-1}(\CHECK K)$.
		If $\phi\in F(V)$ then $\phi(a)=0$ so $\ker\phi\cap B_\varepsilon(a)\neq\emptyset$ so $\phi\notin K$ and thus $F(V)\cap K=\emptyset$ as claimed.
		To finish the proof we just need to show that $F^{-1}(\CHECK K)\subset \widetilde{B_\varepsilon(a)}$.
		Let $\in\Max A$ and suppose $W\notin \widetilde{B_\varepsilon(a)}$. This is equivalent to $d(a,W)\geq\varepsilon$ so
		\cite[Lemma~1.2]{HNO86} there is $\phi\in A^\vee$
		such that $\|\phi\|=1$, $W\subset\ker\phi$ and $d(a,\ker\phi)=d(a,W)\geq\varepsilon$. But then $\ker\phi\cap B_\varepsilon(a)=\emptyset$
		so $\phi\in K\cap F(W)$ so $F(W)\notin\CHECK K$, %
		which completes the proof.\qedhere
	\end{proof}
	
	We can now finish the proof of Theorem~\ref{thm:Fell<=>Fell}.
	
	\begin{corollary}
	  Suppose $A$ is first countable. If $\mathcal A=(X,A,\kappa)$ is a quotient vector bundle with $\kappa$ Fell Continuous
          then $\mathcal A^\vee=(X,A^\vee,\kappa^\vee)$
		is also a quotient vector bundle and $\kappa^\vee$ is also Fell continuous. The converse holds if $A$ is
		a normed vector space.
	\end{corollary}

	\subsection{Normed quotient vector bundles}
	
	We now assume that $A$ is a normed space and $A^\vee$ has the dual norm.
	The objective of this section is to prove the theorem:
	
	\begin{theorem}\label{thm:continuousnormedcodual}
		Let $A$ be a normed vector space, give $A^\vee$ the dual norm,
		consider a map $\kappa\colon X\to\Max A$ and let $\kappa^\vee=F\circ\kappa$.
		\begin{enumerate}
			\item If the triples $(X,A,\kappa)$ and $(X,A^\vee,\kappa^\vee)$ are
			both quotient vector bundles then they are both continuous normed
			quotient vector bundles.
			\item If $A$ is locally uniformly convex and $(X,A^\vee,\kappa^\vee)$
			is a continuous normed quotient vector bundle then $(X,A,\kappa)$ is also
			a continuous normed quotient vector bundle.
			\item If $A$ is reflexive and Fr\'echet smooth and
			$(X,A,\kappa)$ is a continuous normed quotient vector bundle then
			$(X,A^\vee,\kappa^\vee)$ is also a continuous normed quotient vector bundle.
		\end{enumerate}
	\end{theorem}

	\begin{corollary}\label{C7.16}
		Let $\mathcal A=(X,A,\kappa)$ be as above and suppose $A$ is a Banach space and $X$ is first countable and Hausdorff.
		\begin{enumerate}
			\item If $\mathcal A^\vee$ is a quotient vector bundle then both $\mathcal A$ and $\mathcal A^\vee$ are Banach bundles.
			\item If $A$ is reflexive and Fr\'echet smooth then $\mathcal A$ is a Banach bundle if and only if $\mathcal A^\vee$ is a Banach bundle.
		\end{enumerate}
	\end{corollary}

	Given $V\in\Max A$ we denote by $S(V)$ the unit sphere in $V$.
	We will need the following two Lemmas (see \cite[Lemmas~1.1, 1.2]{HNO86}):
	
	\begin{lemma}\label{HNO-1.1}
		Let $a\in A$ and $\phi\in A^\vee$. Then $|\phi(a)|=d(a,\ker\phi)\|\phi\|$.
	\end{lemma}
	
	\begin{lemma}\label{HNO-1.2}
		Let $V\in\Max A$, $V\neq A$ and let $a\in A$.  Then there is $\phi\in A^\vee$ such that $\|\phi\|=1$, $V\subset\ker\phi$ and
		$d(a,V)=d(a,\ker\phi)=\phi(a)$.
	\end{lemma}
	
	We also gather some trivial results about the lower Vietoris topology:
	
	\begin{lemma}\label{factsaboutlV}
		Suppose $\Max A$ has the lower Vietoris topology.
		\begin{enumerate}
			\item For any $\lambda\in\CC$ we have $\widetilde{B_{|\lambda|\varepsilon}(\lambda a)}=\widetilde{B_\varepsilon(a)}$.
			\item Let $U\subset A$ be open and let $V\in\widetilde U$. Then there is $\varepsilon>0$ and $a\in V$ such that
			$\widetilde{B_\varepsilon(a)}\subset\widetilde U$.
			\item Let $V\in\Max A$. The collection $\bigl\{\widetilde {B_\varepsilon(a)}\bigr\}$ with $a\in S(V)$ is a local subbasis of
			neighbourhoods at $V$.
		\end{enumerate}
	\end{lemma}
	
	For the proof of Theorem~\ref{thm:continuousnormedcodual}
	we consider the following four topologies on $\Max A$:
	\begin{enumerate}
		\item The lower Vietoris topology induced by the norm on $A$.
		\item The topology generated by the sets $\mathbf U_r(a)=\{V:d(a,V)>r\}$, with $a\in A$ and $r>0$.
		\item The topology generated by the sets $F^{-1}\bigl(\mathbf U_r(\phi)\bigr)=\{V:d\bigl(\phi,F(V)\bigr)>r\}$, with $\phi\in A^\vee$ and $r>0$.
		\item The topology consisting of the sets $F^{-1}(\mathcal U)$ where $\mathcal U\subset\Max A^\vee$ is open in the lower Vietoris topology
		induced by the dual norm on $A^\vee$.
	\end{enumerate}

	\subsubsection{Topologies 1 and 3}

	\begin{lemma}\label{VinUrphiiffphia>ra}
		Let $V\in\Max A$ and let $\phi\in A^\vee$. The following are equivalent:
		\begin{enumerate}
			\item $d\bigl(\phi,F(V)\bigr)>r$.
			\item $\|\phi|_V\|>r$.
			\item There is $a\in V$ such that $|\phi(a)|-r\|a\|>0$.
		\end{enumerate}
	\end{lemma}
	\begin{proof}
		It is clear that $2\Leftrightarrow3$. We will show that $1\Leftrightarrow2$.
		If $\|\phi|_V\|>r$ then for every $\psi\in F(V)$ we have 
		\[
		\|\psi-\phi\|\geq\sup_{a\in S(V)}|\psi(a)-\phi(a)|=\sup_{a\in S(V)}|\phi(a)|=\|\phi|_V\|
		\]
		so $d\bigl(\phi,F(V)\bigr)\geq\|\phi|_V\|>r$. Conversely, suppose $\|\phi|_V\|\leq r$. 
		By the Hahn-Banach Theorem there is $\hat\phi\in A^\vee$ such that $\hat\phi|_V=\phi|_V$
		and $\|\hat\phi\|\leq r$. Let $\psi=\phi-\hat\phi$. Then $\psi\in F(V)$ and $\|\psi-\phi\|\leq r$
		so $d\bigl(\phi,F(V)\bigr)\leq r$.\qedhere
	\end{proof}
	
	\begin{corollary}\label{normed1finerthan3}
		Topology 1 is finer than topology 3.
	\end{corollary}
	\begin{proof}
		Given $\phi\in A^\vee$ and $r>0$ let $U=\{a\in A:|\phi(a)|-r\|a\|>0\}$. Then $U$ is open and
		$\widetilde U=F^{-1}\bigl(\mathbf U_r(\phi)\bigr)$ since:
		\[
		d\bigl(\phi,F(V)\bigr)>r\Leftrightarrow \exists_{a\in V}|\phi(a)|>r\|a\|
		\Leftrightarrow V\cap U\neq\emptyset\Leftrightarrow V\in\widetilde U\,.\qedhere
		\]
	\end{proof}
	
	\begin{corollary}
		Let $A$ be a normed vector space and let $(X,A,\kappa)$ be a quotient vector bundle. If $(X,A^\vee,\kappa^\vee)$
		is a quotient vecor bundle, then it is a continuous normed quotient vector bundle.
	\end{corollary}
	
	We say $A$ is locally uniformly convex \cite{Lov55} if for any sequence $(a_n)$ in $A$ we have
	\[
	\lim\bigl(2\|a\|^2+2\|a_n\|^2-\|a+a_n\|^2\bigr)=0\ \Rightarrow\ \lim\|a-a_n\|=0\,.
	\]
	
	\begin{lemma}
		If $A$ is locally uniformly convex then topologies 1 and 3 coincide.
	\end{lemma}
	\begin{proof}
		By Corollary~\ref{normed1finerthan3} and Lemma~\ref{factsaboutlV}-3 we only need to show that, for any $a\in A$ and $\varepsilon>0$
		the set $\widetilde B_\varepsilon(a)$ is open in topology 3.
		Let $V\in\widetilde{B_\varepsilon(a)}$. By Lemma~\ref{factsaboutlV}-3 we may assume that $a\in V$ and $\|a\|=1$. 
		By the Hahn-Banach theorem there is
		$\phi\in A^\vee$ such that $\phi(a)=\|a\|=1$ and $\|\phi\|=1$. 
		Then by Lemma~\ref{VinUrphiiffphia>ra} we have $V\in F^{-1}\bigl(\mathbf U_{1-\delta}(\phi)\bigr)$
		whenever $0<\delta<1$. Consider the sets
		$U_\delta=\{b\in S(A):\RE\phi(b)>1-\delta\}$. We claim that there is $\delta>0$ such that $U_\delta\subset B_\varepsilon(a)$.
		We prove by contradiction. For each $n$ pick $a_n$ such that $\|a_n\|=1$, $\RE\phi(a_n)>1-\tfrac1n$ and $\|a_n-a\|\geq\varepsilon$.
		Then $\RE\phi(a+a_n)>2-\frac1n$ so 
		\[
		2-\tfrac1n<|\phi(a+a_n)|\leq\|a+a_n\|\leq 2
		\]
		from which it follows that $\lim\|a+a_n\|=2$.
		Since $A$ is locally uniformly convex we have $\lim\|a_n-a\|=0$, a contradiction. Pick $\delta$ so that $0<\delta<1$ and $U_\delta\subset B_\varepsilon(a)$.
		To finish the proof we only have to show that $F^{-1}\bigl(\mathbf U_{1-\delta}(\phi)\bigr)\subset \widetilde{B_\varepsilon(a)}$. Suppose
		$F(W)\in \mathbf U_{1-\delta}(\phi)$. Then by~\ref{VinUrphiiffphia>ra} there is $b\in S(W)$ such that $|\phi(b)|>1-\delta$. Multiplying by a phase %
		we may assume that $\RE\phi(b)>1-\delta$ so $b\in U_\delta\subset B_\varepsilon(a)$ so $W\cap B_\varepsilon(a)\neq\emptyset$
		which finishes the proof.\qedhere
	\end{proof}
	
	\subsubsection{Topologies 2 and 4}
	
	\begin{lemma}
		Topology 4 is finer than topology 2.
	\end{lemma}
	\begin{proof}
		Let $a\in A$ and let $r>0$. We will show that $\mathbf U_r(a)$ is open in topology 4. Let $V\in\mathbf U_r(a)$, let
		$d=d(a,V)>r$ and choose $r'$ such that $r<r'<d$.
		By the Hahn-Banach separation theorem there is $\phi\in A^\vee$ such that $\phi|_V=0$ and $\RE\phi>0$ on $B_d(a)$.
		\[
		\varepsilon=\frac{\bigl(1-\frac {r'}d\bigr)\RE\phi(a)}{r'+\|a\|}\,.
		\]
		Clearly $V\in F^{-1}\bigl(\widetilde{B_\varepsilon(\phi)}\bigr)$, because $\phi\in F(V)$.
		We claim that $F^{-1}\bigl(\widetilde{B_\varepsilon(\phi)}\bigr)\subset\mathbf U_{r}(a)$. Suppose
		$F(W)\in \widetilde{B_\varepsilon(\phi)}$. Then there is $\psi\in A^\vee$ such that $W\subset\ker\psi$
		and $\|\psi-\phi\|<\varepsilon$.
		Let $b\in B_{r'}(a)$. Then $\bigl|\RE\psi(b)-\RE\phi(b)\bigr|\leq\bigl|\psi(b)-\phi(b)\bigr|<\varepsilon\|b\|$
		and $\|b\|<r'+\|a\|$ so
		\[
		\RE\psi(b)>\RE\phi(b)-\varepsilon\|b\|
		>\RE\phi(b)-\varepsilon(r'+\|a\|)
		=\RE\phi(b)-\bigl(1-\frac {r'}d\bigr)\RE\phi(a)
		\]
		We claim that $\RE\psi(b)>0$. Since $b\in B_{r'}(a)$ we have
		$\tfrac{d}{r'}(b-a)+a\in B_d(a)$ so $\RE\phi\bigl(\frac d{r'}(b-a)+a\bigr)>0$. It follows that
		$\RE\phi(b)>\bigl(1-\frac {r'}d\bigr)\RE\phi(a)$ and it follows that $\RE\psi(b)>0$.
		But then $\psi(b)\neq 0$ for all $b\in B_{r'}(a)$ so $W\cap B_{r'}(a)=0$ so $d(a,W)\geq r'>r$ which finishes the proof.\qedhere
	\end{proof}
	
	\begin{corollary}
		Let $A$ be a normed vector space and let $(X,A,\kappa)$ be a quotient vector bundle. If the map $\kappa^\vee$ is continuous
		then $(X,A,\kappa)$ is a continuous normed quotient vector bundle.
	\end{corollary}

	Before we turn to the next result we need an easy convexity result:
	
	\begin{lemma}\label{convexitilambda>1}
		Let $V\in\Max A$, let $a\in A$ be such that $\|a\|=d(a,V)$, let $b\in V$ and let $\lambda$ be a positive real number.
		If $\|a-\lambda b\|>\|a-b\|$ then $\lambda>1$.
	\end{lemma}
	\begin{proof}
		Suppose $\lambda\leq1$. We have $\|a\|=d(a,V)\leq\|a-b\|$ so:
		\begin{align*}
		\|a-\lambda b\|&=\|\lambda(a-b)+(1-\lambda)a\|\\
		&\leq\lambda\|a-b\|+(1-\lambda)\|a\|\\
		&\leq\lambda\|a-b\|+(1-\lambda)\|a-b\|=\|a-b\|\,.\qedhere
		\end{align*}
	\end{proof}
	
	We say the norm is Fr\'echet differentiable at a point $a\in S(A)$ if  the limit
	\[
	\lim_{\delta\to 0}\frac{\|a-\delta b\|-\|a\|}\delta
	\]
	exists uniformly in $b\in S(A)$. We say $A$ is Fr\'echet smooth if its norm is Fr\'echet differentiable at every $a\in S(A)$.
	
	\begin{lemma}
		If $A$ is reflexive and Fr\'echet smooth then topologies 2 and 4 coincide.
	\end{lemma}
	\begin{proof}
		Suppose $F(V)\in\widetilde{B_\varepsilon(\phi)}$. By Lemma~\ref{factsaboutlV}-3 we may assume that $\phi\in F(V)$ and $\|\phi\|=1$.
		Since $A$ is reflexive there is $a\in A$ such that $\|a\|=d(a,\ker\phi)=d(a,V)$. Then $|\phi(a)|=d(a,\ker\phi)\|\phi\|=\|a\|$
		so we may assume that $\phi(a)=\|a\|=1$.
		Since $\|a\|$ minimizes $\|a-b\|$ with $b\in\ker\phi$, and $A$ is Fr\'echet smooth, the Fr\'echet derivative of the norm at $a$ along any $b\in\ker\phi$ is zero:
		\[
		\lim_{\delta\to0}\frac{\|a-\delta b\|-\|a\|}\delta=0\quad\text{uniformly on $b\in S(\ker\phi)$.}
		\]
		Let $\delta>0$ be such that, for all $b\in S(\ker\phi)$ we have $\bigl(\|a-\delta b\|-\|a\|\bigr)/\delta<\frac\varepsilon6$.
		Let $\lambda=\frac\varepsilon3+\frac1\delta$. Then for any $b\in S(\ker\phi)$ we have:
		\[
		\frac{\|a-\delta b\|-\|a\|}\delta=\|(\lambda-\tfrac\varepsilon3)a-b\|-(\lambda-\tfrac\varepsilon3)<\tfrac\varepsilon6
		\]
		so if we let $r=\lambda-\frac\varepsilon6$ then 
		\begin{equation}\label{eq:lambda-eps>r}
		\|(\lambda-\tfrac\varepsilon3)a-b\|<r\,.
		\end{equation}
		Now $d(\lambda a,V)=\lambda>r$
		so $V\in\mathbf U_r(\lambda a)$. We want to show that $F\bigl(\mathbf U_r(\lambda a)\bigr)\subset\widetilde{B_\varepsilon(\phi)}$.
		Let $W\in\Max A$ be such that $d(\lambda a,W)>r$ and let $d=d(\lambda a,W)$.
		By the Hahn-Banach theorem there is $\psi\in A^\vee$ such that $W\subset\ker\psi$ (and hence $\psi\in F(W)$), $d(\lambda a,\ker\psi)=d(\lambda a,W)>r$ and $\psi(a)=1$.
		We claim that $\psi\in B_\varepsilon(\phi)$, and hence $F(W)\in \widetilde{B_\varepsilon(\phi)}$. To prove the claim
		consider the projection $p\colon A\to\ker\phi$ defined by $p(b)=b-\phi(b)a$. Then $\|p(b)\|\leq\|b\|+\|\phi\|\|b\|\|a\|=2\|b\|$ so $\|p\|\leq 2$.
		To finish the proof we just need to show that $\|\psi|_{\ker\phi}\|\leq\frac\varepsilon3$ because it will follow that, for any $b\in A$:
		\[
		|\psi(b)-\phi(b)|=|\psi(p(b))|\leq\tfrac\varepsilon3\|p(b)\|\leq\tfrac23\varepsilon\|b\| %
		\]
		and hence $\|\psi-\phi\|\leq\frac23\varepsilon<\varepsilon$.
		So let $b\in S(\ker\phi)$ and suppose $\psi(b)\neq 0$. By multiplying $b$ by a phase we may assume that $\psi(b)$ is real and negative. Let $x=-\psi(b)>0$.
		Then $\frac\varepsilon3 a+\frac\varepsilon{3x}b\in\ker\psi$ and since $d(\lambda a,\ker\psi)>r$, using equation~\eqref{eq:lambda-eps>r} we get
		\[
		\|(\lambda-\tfrac\varepsilon3)a-b\|<r<\|\lambda a-(\tfrac\varepsilon3 a+\tfrac\varepsilon{3x}b)\|=\|(\lambda-\tfrac\varepsilon3)a-\tfrac\varepsilon{3x}b\|\,.
		\]
		We now apply Lemma~\ref{convexitilambda>1} to conclude that $\frac\varepsilon{3x}>1$, so $x=|\psi(b)|<\frac\varepsilon3$.
		This completes the proof.\qedhere
	\end{proof}
	
	\begin{corollary}
		Let $A$ be a reflexive, Fr\'echet smooth and locally uniformly convex normed vector space. Then a triple
		$(X,A,\kappa)$ is a continuous normed quotient vector bundle if and only if $(X,A^\vee,\kappa^\vee)$ is a continuous quotient vector bundle.
	\end{corollary}

\bibliography{dual}
\bibliographystyle{amsplain}

\end{document}